\documentclass[a4paper]{amsart}

\usepackage{amssymb}
\usepackage{amsthm}
\usepackage{amsmath}
\usepackage{braket}
\usepackage[all]{xy}

\usepackage[
left = 20truemm, 
right = 20truemm
]{geometry}

\title{Stark Systems over Gorenstein Local Rings}
\author{Ryotaro Sakamoto}
\date{\today}
\address{Graduate School of Mathematical Sciences, The University of Tokyo, 
3-8-1 Komaba, Meguro-Ku, Tokyo, 153-8914, Japan}
\email{sakamoto@ms.u-tokyo.ac.jp}
\begin{document}
\maketitle

\begin{abstract}
In this paper, we define a Stark system over a complete Gorenstein local ring with a finite residue field. 
Under some standard assumptions, we show that the module of Stark systems is free of rank one and 
that these systems control all the Fitting ideals of the Pontryagin dual of the dual Selmer group. 
This is a generalization of the theory, developed by B. Mazur and K. Rubin, 
on Stark (or Kolyvagin) systems over principal ideal local rings. 
Applying our result to a certain Selmer structure over the cyclotomic Iwasawa algebra, 
we suggest a new method for controlling Selmer groups using Euler systems. 
\end{abstract}

\newtheorem{theorem}{Theorem}[section]
\newtheorem{proposition}[theorem]{Proposition}
\newtheorem{lemma}[theorem]{Lemma}
\newtheorem{corollary}[theorem]{Corollary}

\theoremstyle{remark}
\newtheorem{remark}[theorem]{Remark}

\theoremstyle{definition}
\newtheorem{definition}[theorem]{Definition}
\newtheorem{example}[theorem]{Example}
\newtheorem{hypothesis}[theorem]{Hypothesis}

\newcommand{\Hom}{{\rm Hom}}
\newcommand{\Gal}{{\rm Gal}}

\newcommand{\bC}{\mathbb{C}}
\newcommand{\bF}{\mathbb{F}}
\newcommand{\bQ}{\mathbb{Q}}
\newcommand{\bL}{\mathbb{L}}
\newcommand{\bR}{\mathbb{R}}
\newcommand{\bT}{\mathbb{T}}
\newcommand{\bZ}{\mathbb{Z}}

\newcommand{\cA}{\mathcal{A}}
\newcommand{\cB}{\mathcal{B}}
\newcommand{\cC}{\mathcal{C}}
\newcommand{\cD}{\mathcal{D}}
\newcommand{\cE}{\mathcal{E}}
\newcommand{\cF}{\mathcal{F}}
\newcommand{\cG}{\mathcal{G}}
\newcommand{\cH}{\mathcal{H}}
\newcommand{\cI}{\mathcal{I}}
\newcommand{\cJ}{\mathcal{J}}
\newcommand{\cK}{\mathcal{K}}
\newcommand{\cL}{\mathcal{L}}
\newcommand{\cM}{\mathcal{M}}
\newcommand{\cN}{\mathcal{N}}
\newcommand{\cO}{\mathcal{O}}
\newcommand{\cP}{\mathcal{P}}
\newcommand{\cQ}{\mathcal{Q}}
\newcommand{\cS}{\mathcal{S}}
\newcommand{\cT}{\mathcal{T}}

\newcommand{\fa}{\mathfrak{a}}
\newcommand{\fb}{\mathfrak{b}}
\newcommand{\fc}{\mathfrak{c}}
\newcommand{\fd}{\mathfrak{d}}
\newcommand{\ff}{\mathfrak{f}}
\newcommand{\fq}{\mathfrak{q}}
\newcommand{\fp}{\mathfrak{p}}
\newcommand{\fn}{\mathfrak{n}}
\newcommand{\fm}{\mathfrak{m}}

\section{Introduction}

Euler systems had been introduced by V.A. Kolyvagin in order to study Selmer groups. 
He constructed a collection of cohomology classes 
which is called Kolyvagin's derivative classes from an Euler system. 
He used these classes which come from the Euler system of Heegner points to bound the order of the Selmer group of 
certain elliptic curves over imaginary quadratic fields. 
By using the Euler system of the cyclotomic units, K. Rubin gave another proof of 
the classical Iwasawa Main Conjecture over $\bQ$, which had been proved by B. Mazur and A. Wiles. 
After that many mathematicians, especially K. Kato, B. Perrin-Riou, and K. Rubin, have studied 
Selmer gruops using Euler systems. 

B. Mazur and K. Rubin noticed that Kolyvagin's derivative classes have good interrelations.  
In \cite{MR1}, they defined a Kolyvagin system 
to be a system of cohomology classes satisfying these good interrelations. 

To explain their works in \cite{MR1} and \cite{MR2}, we introduce some notation.  
Let $K$ be a number field and $G_K$ denote the absolute Galois group of $K$. 
Let $R$ be a complete noetherian local ring with a finite residue field of odd characteristic and 
$T$ be a free $R$-module with an $R$-linear continuous $G_K$-action which is unramified 
outside a finite set of places of $K$. 
Let $\cF$ be a Selmer structure on $T$. 
Suppose that $T$ and $\cF$ satisfy some conditions (See Section 3 and Section 4). 
Let $\chi(\cF) \in \bZ$ be the core rank of $\cF$ (See Definition \ref{core rank}). 

Suppose that $R$ is a principal ideal local ring and $\chi(\cF) = 1$. 
In \cite{MR1}, they proved that the module of Kolyvagin systems for $\cF$ is free of rank one and 
that the $R$-module structure of the dual Selmer group associated with $\cF$ is determined by these systems. 
Furthermore, if $R$ is the ring of integers in a finite extension of the field $\bQ_p$, 
then they constructed a canonical morphism from the module of Euler systems to 
the module of Kolyvagin systems for the canonical Selmer structure (See Example \ref{cansel}). 

We still assume that $R$ is a principal ideal local ring. Suppose that $\chi(\cF) > 0$. 
In \cite{Sa} and \cite{MR2}, T. Sano and B. Mazur - K. Rubin defined Stark systems, independently 
(T. Sano called them unit systems). 
In \cite{MR2}, B. Mazur and K. Rubin removed the assumption $\chi(\cF) = 1$ 
by using Stark systems instead of Kolyvagin systems. 
Namely, they proved that the $R$-module structure of the dual Selmer group associated with $\cF$ 
is determined by Stark systems. 
Furthermore, they showed that the module of Stark systems is free of rank one and 
that there is an isomorphism from the module of Stark systems 
to the module of Kolyvagin systems if $\chi(\cF)=1$. 

To prove Iwasawa Main Conjecture using an Euler system, 
we need a Stark system over a complete regular local ring. 
In \cite{MR2}, the elementary divisor theorem is used to define the module of Stark systems. 
Hence Stark systems are not defined in \cite{MR2} when $R$ is not a principal ideal ring. 
Thus clearly we have two important problems. 
The first problem is to define a Stark system over an arbitrary complete noetherian local ring. 
The second problem is to control Selmer groups using Stark systems. 
For these problems, we show the following results in this paper. 

\subsection*{Stark systems over Gorenstein local rings}
Suppose that $R$ is a zero dimensional Gorenstein local ring. 
To define a Stark system over $R$, we use a natural generalization of Rubin's lattice, 
called the exterior bidual, instead of the exterior power. 
In \cite{BKS}, D. Burns, M. Kurihara, and T. Sano 
used the exterior bidual to state some conjectures about Rubin-Stark elements. 

By using the exterior bidual, we generalize 
some linear algebraic lemmas in \cite{MR2} to zero dimensional Gorenstein local rings (See Section 2).  
Thus, in the same way as \cite{MR2}, 
we are able to define a Stark system over $R$. 
Furthermore, under some mild assumptions, we prove that 
the module of Stark systems over $R$ is free of rank one (Theorem \ref{core}). 

Suppose that $R$ is an arbitrary complete Gorenstein local ring. 
Then we can construct an inverse system of the modules of Stark systems 
over some zero dimensional Gorenstein quotients of $R$. 
We define the module of Stark systems over $R$ to be its inverse limit and 
also prove that this module is free of rank one (Theorem \ref{main1}). 

\begin{remark}
When $R$ is the cyclotomic Iwasawa algebra, $\cF$ is the canonical Selmer structure, and $\chi(\cF) =1$, 
K. B\"uy\"uboduk proved that the module of Kolyvagin systems for $\cF$ is free of rank one in \cite{Bu}. 
More generally, when $R$ is a certain Gorenstein local ring, $\cF$ is a certain Selmer structure, and $\chi(\cF) = 1$, 
he showed that the module of Kolyvagin systems for $\cF$ is free of rank one in \cite{buk}. 
In this case, we are able to construct an isomorphism from the module of Kolyvagin systems to 
the module of Stark systems. Hence Theorem \ref{main1} is a generalization of this result. 
\end{remark}

\subsection*{Controlling Selmer groups using Stark systems}
Under some mild assumptions, we prove that all the Fitting ideals of the Pontryagin dual 
of the dual Selmer group associated with $\cF$ is determined by Stark systems for $\cF$ 
when $R$ is a complete Gorenstein local ring (Theorem \ref{main2}). 

\begin{remark}
If $R$ is a principal ideal ring and $M$ is a finitely generated $R$-module, 
then there is a non-canonical isomorphism $M \simeq \Hom(M, \bQ_p/\bZ_p)$ as $R$-modules, 
which implies that 
\begin{align*}
{\rm Fitt}_{i, R}(M) = {\rm Fitt}_{i, R}\left(\Hom(M, \bQ_p/\bZ_p)\right)
\end{align*} 
for any non-negative integer $i$. Thus Theorem \ref{main2} is a generalization of \cite[Theorem 8.6]{MR2} 
(See Proposition \ref{gen} and Remark \ref{compare-dvr}). 
\end{remark}

Applying our result to a certain Selmer structure over the cyclotomic Iwasawa algebra $\Lambda$, 
we prove that all the Fitting ideals of the Pontryagin dual of 
the $\Lambda$-adic dual Selmer group is determined by $\Lambda$-adic Stark systems (Theorem \ref{lambda-stark}). 
In particular, we show that $\Lambda$-adic Stark systems determine not only the characteristic ideal of 
this $\Lambda$-module but also its pseudo-isomorphism class. 
In \cite{MR1}, they showed that a primitive $\Lambda$-adic Kolyvagin system control the characteristic ideal of 
the Pontryagin dual of the $\Lambda$-adic dual Selmer group when the core rank is one. 
We will show that Theorem \ref{lambda-stark} is a generalization of this result (Proposition \ref{mr}). 
However, our proof of Theorem \ref{lambda-stark} is different from the proof in \cite{MR1}. 
Furthermore, there is a canonical map from 
the module of Euler systems to the module of $\Lambda$-adic Stark systems when the core rank is one. 
Hence we suggest a new method for controlling Selmer groups using Euler systems. 

\begin{remark}
After the author had got almost all the results in this paper, 
T. Sano told the author that D. Burns and he also were studying Stark systems and Kolyvagin systems 
over zero dimensional Gorensitein rings using an exterior bidual (See \cite{BS}). 
\end{remark}

\subsection*{Acknowledgement}
The author would like to express his gratitude to his supervisor Professor Takeshi Tsuji for many helpful advices. 
The author is supported by the FMSP program at the University of Tokyo. 

\subsection*{Notation} 
Let $p$ be an odd prime number and $K$ a number field. 
Let $K_{v}$ denote the completion of $K$ at a place $v$. 
For a field $L$, let $\overline{L}$ denote a fixed separable closure of $L$ and 
$G_{L} := \Gal(\overline{L}/L)$ the absolute Galois group of $L$. 

Throughout this paper, $R$ denotes a complete noetherian local ring 
with a finite residue field $\Bbbk$ of characteristic $p$. 
Let $\fm_R$ be the maximal ideal of $R$. 
In addition, $T$ denotes a free $R$-module of finite rank 
with an $R$-linear continuous $G_{K}$-action which is unramified outside a finite set of places of $K$. 
For any non-negative integer $i$, 
let $H^i(K, T)$ be the $i$-th continuous cohomology group of $G_K$ with coefficient $T$. 

If $A$ is a commutative ring and $I$ is an ideal of $A$, 
then define 
\begin{align*}
M^* := \Hom_{A}(M,A)
\end{align*}
and $M[I] := \{ x \in M \mid Ix = 0 \}$ for any $A$-module $M$. 
When $I = (x_1, \ldots, x_n)$, we write $M[x_1, \ldots, x_n]$ instead of $M[I]$. 
Let $\bigwedge^r_A M$ denote the $r$-th exterior power of $M$ in the category of $A$-modules. 
We denote by $L_A (M) \in \bZ_{>0} \cup \{\infty\}$ the $A$-length of $M$. 

\section{Preliminaries}
Let $A$ be a commutative ring, $M$ an $A$-module and $r$ a non-negative integer. 
We define the $r$-th exterior bidual $\bigcap^r_A M$ of $M$ by 
\begin{align*}
\bigcap^{r}_A M := \left( \bigwedge^{r}_A M^* \right)^* 
= \Hom_A \left( \bigwedge^r_A \Hom_A (M, A), A \right). 
\end{align*}
The pairing 
\begin{align*}
\bigwedge^{r}_A M \times \bigwedge^{r}_A M^* \to A, \ 
(m_{1} \wedge \cdots \wedge m_{r}, f_{1} \wedge \cdots \wedge f_{r}) \mapsto {\rm det}(f_{i}(m_{j})) 
\end{align*}
induces a canonical homomorphism 
\begin{align*}
l_{M} \colon \bigwedge^{r}_A M \to \bigcap^{r}_A M.  
\end{align*}
Note that the map $l_{M}$ is an isomorphism if $M$ is a finitely generated projective $A$-module. 

\begin{lemma}\label{base}
Let $F \xrightarrow{h} M \xrightarrow{g} N \to 0$ 
be an exact sequence of $A$-modules and $r$ a non-negative integer. 
If $F$ is a free $A$-module of rank $s \leq r$, then there is a unique $A$-linear map 
\begin{align*}
\phi \colon \bigwedge^{r-s}_A N \otimes_A \det(F) \to \bigwedge^{r}_A M
\end{align*}
such that $\phi\left( \wedge^{r-s}g(a) \otimes b \right) = a \wedge \left( \wedge^{s}h (b) \right)$ for any 
$a \in \bigwedge^{r-s}_A M$ and $b \in \det(F)$. 
\end{lemma}
\begin{proof}
Let $n \in \bigwedge^{r-s}_A N$ and $b \in \det(F)$. 
Since the map $\bigwedge^{r-s}_A M \to \bigwedge^{r-s}_A N$ is surjective, 
we can take a lift $m$ of $n$ in $\bigwedge^{r-s}_A M$. 
Then we define the map $\phi$ by $\phi(n \otimes b) = m \wedge (\wedge^{s}h (b))$. 
Since $\bigwedge^{s+1}_A {\rm im}(h) = 0$, the map $\phi$ is well-defined. 
\end{proof}

\begin{lemma}\label{commlem}
Let $r$ be a non-negative integer. 
Suppose that we have the following commutative diagram of $A$-modules:
\begin{align*}
\xymatrix{ 
F \ar[r]^{\alpha} \ar[d]^{h} & M \ar[r]^{\beta} \ar[d]^{f} & N \ar[r] \ar[d]^{g} & 0
\\
F' \ar[r]^{\alpha'} & M' \ar[r]^{\beta'} & N' \ar[r] & 0
}
\end{align*}
where the horizontal rows are exact and both $F$ and $F'$ are free $A$-modules of rank $s \leq r$. 
Then the following diagram is commutative: 
\begin{align*}
\xymatrix{
\bigwedge^{r-s}_A N \otimes_A \det(F) \ar[r]^-{\phi} \ar[d]^{\wedge^{r-s}g \otimes \wedge^{s} h} 
& \bigwedge^{r}_A M \ar[d]^{\wedge^r f} 
\\
\bigwedge^{r-s}_A N' \otimes_A \det(F') \ar[r]^-{\phi'} & \bigwedge^{r}_A M' 
}
\end{align*}
where $\phi$ and $\phi'$ are the maps defined in Lemma \ref{base}. 
\end{lemma}
\begin{proof}
Let $a \in \bigwedge^{r-s}_A M$ and $b \in \det(F)$. 
By Lemma \ref{base}, we compute 
\begin{align*}
\phi' \circ (\wedge^{r-s}g \otimes \wedge^{s} h)(\wedge^{r-s}\beta(a) \wedge b) 
&= \phi'( \wedge^{r-s} (g \circ \beta)(a) \otimes \wedge^s h(b))
\\
&= \phi'( \wedge^{r-s} (\beta' \circ f)(a) \otimes \wedge^s h(b))
\\
&= \wedge^{r-s}f(a) \otimes (\wedge^s \alpha' \circ h)(b)
\\
&= \wedge^{r}f ( a \wedge (\wedge^s \alpha(b)))
\\
&= \wedge^{r}f \circ \phi( \wedge^{r-s}\beta(a) \otimes b). 
\end{align*}
This completes the proof. 
\end{proof}

Suppose that $A$ is a zero dimensional Gorenstein local ring. 
Note that all the free $A$-modules are injective by the definition of Gorenstein ring. 
In particular, the functor $(-)^* = \Hom_A(-, A)$ is exact. 
Furthermore, Matlis duality assert that $L_A(M) = L_A(M^*)$ and 
the canonical map $l_M \colon M \to \bigcap^1_A M$ is an isomorphism 
for any finitely generated $A$-module $M$. 

For $i \in \{1, 2\}$, let $M_{i}$ be an $A$-module and $F_{i}$ a free $A$-module of rank $s_{i}$. 
Suppose that the following diagram is cartesian:
\begin{align*}
\xymatrix{
M_{1} \ar@{^{(}->}[r] \ar[d] & M_{2} \ar[d] 
\\
F_{1} \ar@{^{(}->}[r] &  F_{2} 
}
\end{align*}
where the horizontal maps are injective. 
Note that $F_2/F_1$ is a free $A$-module of rank $s_2 - s_1$. 
Applying Lemma \ref{base} to the exact sequence 
$(F_{2}/F_1)^* \to M_2^* \to M_1^* \to 0$, we obtain a map 
\begin{align*}
\widetilde{\Phi} \colon \bigcap^r_A M_2 \otimes_A \det((F_2/F_1)^*) \to \bigcap^{r-s_2+s_1}_A M_1
\end{align*}
for any non-negative integer $r \geq s_2 - s_1$. 
Therefore we get the following map which play an important role in this paper. 

\begin{definition}\label{map}
For any cartesian diagram as above and a non-negative integer $r \geq s_2 - s_1$, we define a map 
\begin{align*}
\Phi \colon \bigcap^r_A M_2 \otimes_A \det(F_2^*) \to 
\bigcap^r_A M_2 \otimes_A \det((F_2/F_1)^*) \otimes_A \det(F_1^*) 
\\
\xrightarrow{\widetilde{\Phi} \otimes {\rm id}_{\det(F_1^*)}} \bigcap^{r-s_2+s_1}_A M_1 \otimes_A \det(F_1^*)
\end{align*}
where the first map is induced by an isomorphism 
$\det((F_2/F_1)^*) \otimes_A \det(F_1^*) \simeq \det(F_2^*)$, 
$a \otimes b \to a \wedge \widetilde{b}$. 
Here $\widetilde{b}$ is a lift of $b$ in $\bigwedge^{s_1}_A F_2^*$. 
\end{definition}

The map $\Phi$ is a generalization of the map defined in \cite[Proposition A.1]{MR2}. 
Furthermore, we have the following proposition which is a generalization of \cite[Proposition A.2]{MR2} 
to zero dimensional Gorenstein local rings. 

\begin{proposition}\label{homlem}
Let $A$ be a zero dimensional Gorenstein local ring. 
Suppose that we have the following commutative diagram of $A$-modules
\begin{align*}
\xymatrix{
M_{1} \ar@{^{(}->}[r] \ar[d] & M_{2} \ar[d] \ar@{^{(}->}[r] & M_3 \ar[d] 
\\
F_{1} \ar@{^{(}->}[r] &  F_{2} \ar@{^{(}->}[r] & F_3
}
\end{align*}
such that $F_1$, $F_2$, and $F_3$ are free of finite rank and the two squares are cartesian. 
Let $s_i = {\rm rank}_A (F_i)$ and  $r \geq s_3 - s_1$.  
If $i,j \in \{1,2,3\}$ and $j < i$, we denote by
\begin{align*}
\Phi_{ij} \colon \bigcap^{r-s_3 +s_i}_A M_{i} \otimes_A {\rm det}(F_{i}^*) \to 
\bigcap^{r-s_3 + s_j}_A M_{j} \otimes_A {\rm det}(F_{j}^*)
\end{align*}
the map given by Definition \ref{map}. Then we have 
\begin{align*}
(-1)^{(s_2-s_1)(s_3-s_2)} \Phi_{31} = \Phi_{21} \circ \Phi_{32}. 
\end{align*}
\end{proposition}
\begin{proof}
Let $i \in \{1, 2\}$. 
Applying Lemma \ref{base} to the exact sequence 
$(F_{i+1}/F_i)^* \xrightarrow{h_i} M_{i+1}^* \to M_i^* \to 0$, we have a map 
\begin{align*}
\phi_i \colon \bigwedge^{r-s_3+s_i}_A M_i^* \otimes_A \det((F_{i+1}/F_i)^*) \to 
\bigwedge^{r - s_3 + s_{i+1}}_A M_{i+1}^*. 
\end{align*}
Let $x_1 \in \bigwedge^{r- s_3 + s_1}_A M_1^*$, 
$y \in \det((F_2/F_1)^*)$, and $z \in \det((F_3/F_2)^*)$. 
Then we compute 
\begin{align*}
\phi_2 \circ (\phi_1 \otimes {\rm id}_{\det((F_3/F_2)^*)})(x_1 \otimes y \otimes z) 
&= \phi_2((x_2 \wedge (\wedge^{s_2-s_1}h_1(y))) \otimes z)
\\
&= x_3 \wedge w \wedge (\wedge^{s_3-s_2}h_2(z))
\end{align*}
where $x_{2}$ is a lift of $x_1$ to $\bigwedge^{r-s_3+s_1}_A M_{2}^*$, 
$x_{3}$ is a lift of $x_2$ to $\bigwedge^{r-s_3+s_1}_A M_{3}^*$, 
and $w$ is a lift of $\wedge^{s_2-s_1}h_1(y)$ to $\bigwedge^{s_2-s_1}_A M_3^*$. 
Let $\widetilde{y}$ be a lift of $y$ to $\bigwedge^{s_2-s_1}_A (F_3/F_1)^*$. 
We denote by $h_3 \colon (F_3/F_1)^* \to M_3^*$ the dual of the map $M_3 \to F_3/F_1$. 
Since $\wedge^{s_2-s_1}h_3(\widetilde{y})$ is a lift of $\wedge^{s_2-s_1}h_1(y)$, 
we may assume that $w = \wedge^{s_2-s_1}h_3(\widetilde{y})$. 
Let 
\begin{align*}
\phi_3 \colon \bigwedge^{r-s_3+s_1}_A M_1 \otimes_A \det((F_3/F_1)^*) \to \bigwedge^{r}_A M_3^*
\end{align*}
be the map defined by the exact sequence $(F_3/F_1)^* \xrightarrow{h_3} M_3^* \to M_1^* \to 0$ 
using Lemma \ref{base}. Then we have 
\begin{align*}
\phi_{3}(x_1 \otimes (z \wedge \widetilde{y})) &= x_3 \wedge (\wedge^{s_3-s_2}h_3(z)) \wedge (\wedge^{s_2-s_1}h_3(\widetilde{y}))
\\
&= (-1)^{(s_2-s_1)(s_3-s_2)} x_3 \wedge w \wedge (\wedge^{s_3-s_2}h_2(z)). 
\end{align*}
Note that the image of $y \otimes z$ under the isomorphism 
\begin{align*}
\det((F_2/F_1)^*) \otimes_A \det((F_3/F_2)^*) \simeq \det((F_3/F_1)^*)
\end{align*}
in Definition \ref{map} is $z \wedge \widetilde{y}$. Hence by taking dual, we get the desired equality. 
\end{proof}

\begin{proposition}\label{commlem2}
Let $A$ be a zero dimensional Gorenstein local ring and $r$ a non-negative integer. 
Suppose that we have the following commutative diagram of $A$-modules:
\begin{align*}
\xymatrix{
 0 \ar[r] & N \ar[r] \ar[d] & M \ar[r] \ar[d] & F \ar[d]
\\
 0 \ar[r] & N' \ar[r] & M' \ar[r] & F'
}
\end{align*}
where the horizontal rows are exact and both $F$ and $F'$ are free $A$-modules of rank $s \leq r$. 
Then the following diagram commutes: 
\begin{align*}
\xymatrix{
 \bigcap^{r}_A M  \ar[r] \ar[d] & \bigcap^{r-s}_A N \otimes_A \det(F) \ar[d]
\\
 \bigcap^{r}_A M' \ar[r] & \bigcap^{r-s}_A N' \otimes_A \det(F') 
}
\end{align*}
where the horizontal maps are defined by Definition \ref{map}. 
\end{proposition}
\begin{proof}
This proposition follows from Lemma \ref{commlem}.  
\end{proof}

\section{Selmer Structures}

In this section, we review the results of \cite{MR1} and \cite{MR2}. 
Recall that $p$ is an odd prime number, $K$ is a number field, 
$(R, \fm_R)$ is a complete noetherian local ring with finite residue field $\Bbbk$ of characteristic $p$, 
and $T$ is a free $R$-module of finite rank with an $R$-linear continuous $G_{K}$-action 
which is unramified outside a finite set of places of $K$. 

Throughout this paper, 
we assume that the field $\overline{K}$ is contained in the complex number field $\bC$ and 
that the fixed separable closure $\overline{K}_\fq$ of $K_\fq$ containing $\overline{K}$ for each prime $\fq$ of $K$. 
Let $K(\fq)$ denote the $p$-part of the ray class field of $K$ modulo $\fq$ and 
$K(\fq)_\fq$ the closure of $K(\fq)$ in $\overline{K}_\fq$.
Let $\cD_\fq := \Gal(\overline{K}_\fq/ K_\fq)$ be the decomposition group at $\fq$ in $G_K$.  
Denote $H^1(K_\fq, T) = H^1(\cD_\fq, T)$.  
Let $\cI_\fq \subseteq \cD_\fq$ be the inertia group at $\fq$ and 
${\rm Fr}_\fq$ the Frobenius element of $\cD_\fq / \cI_\fq$. 
We write ${\rm loc}_\fq$ for the localization map $H^1(K, T) \to H^1(K_\fq, T)$. 
Define 
\begin{align*}
H^1_{\rm ur}(K_\fq, T) := \ker \left( H^{1}(K_{\fq}, T) \to H^{1}(\cI_\fq, T) \right)
\end{align*}
and 
\begin{align*}
H^1_{\rm tr}(K_\fq, T) := \ker \left( H^{1}(K_{\fq}, T) \to H^{1}(K(\fq)_\fq, T) \right)
\end{align*}
for each prime $\fq$ of $K$. 
Furthermore, we set $H^1_f(K_\fq, T) = H^1_{\rm ur}(K_\fq, T)$ if $T$ is unramified at a prime $\fq$ of $K$. 

\begin{definition}
A Selmer structure $\cF$ on $T$ is a collection of the following data:
\begin{itemize}
\item a finite set $\Sigma(\cF)$ of places of $K$, including all the infinite places, all the primes above $p$, 
and all the primes where $T$ is ramified,
\item a choice of $R$-submodule $H^{1}_{\cF}(K_{\fq}, T) \subseteq H^{1}(K_{\fq},T)$ for each $\fq \in \Sigma(\cF)$. 
\end{itemize}
Put $H^{1}_{\cF}(K_{\fq},T) = H^{1}_{\rm ur}(K_{\fq},T)$ for each prime $\fq \not\in \Sigma(\cF)$. 
We call $H^{1}_{\cF}(K_{\fq}, T)$ the local condition of $\cF$ at a prime $\fq$ of $K$. 
\end{definition}

Since $p$ is an odd prime, we have $H^1(K_v, T) = 0$ for any infinite place $v$ of $K$. 
Thus we ignore local condition at any infinite place of $K$.  

\begin{remark}
Let $\cF$ be a Selmer structure on $T$ and $R \to S$ a surjective ring homomorphism. 
Then $\cF$ induces a Selmer structure on the $S$-module $T \otimes_R S$, that we will denote by $\cF_S$. 
If there is no risk of confusion, then we write $\cF$ instead of $\cF_S$. 
\end{remark}

\begin{definition}
Let $\cF$ be a Selmer structure on $T$. 
Set
\begin{align*}
H^1_{/\cF}(K_\fq, T) := H^1(K_\fq, T)/H^1_{\cF}(K_\fq, T)
\end{align*} 
for any prime $\fq$ of $K$. 
We define the Selmer group $H_{\cF}^{1}(K, T) \subseteq H^{1}(K, T)$ associated with $\cF$ 
to be the kernel of the direct sum of localization maps
\begin{align*}
H_{\cF}^{1}(K,T) := \ker \left( H^{1}(K,T) \to \bigoplus_{\fq} 
H^{1}_{/\cF}(K_{\fq},T) \right)
\end{align*}
where $\fq$ runs through all the primes of $K$. 
\end{definition}

\begin{example}\label{cansel}
Suppose that $R$ is $p$-torsion-free. 
Then we define a canonical Selmer structure $\cF_{\rm can}$ on $T$ by the following data:
\begin{itemize}
\item $\Sigma(\cF_{\rm can}) = \{ \fq \mid T \text{ is ramified at } \fq \} \cup \{\fp \mid p \} 
\cup \{v \mid \infty \}$,  
\item we define a local condition at a prime $\fq \nmid p$ to be 
\begin{align*}
H^{1}_{\cF_{\rm can}}(K_{\fq}, T) = H^1_f(K_\fq, T) 
:= \ker \left( H^{1}(K_{\fq}, T) \to H^{1}(\cI_\fq, T \otimes_{\bZ_p} \bQ_p) \right),  
\end{align*}
\item we define a local condition at a prime $\fp \mid p$ to be 
$H^{1}_{\cF_{\rm can}}(K_{\fp}, T) = H^{1}(K_{\fp}, T)$. 
\end{itemize}
Note that $H^1_{\rm ur}(K_\fq, T) = H^{1}_f(K_{\fq}, T)$ for any prime $\fq$ where $T$ is unramified 
since the map $H^1(\cI_\fq, T) = \Hom(\cI_\fq, T) \to \Hom(\cI_\fq, T \otimes_{\bZ_p} \bQ_p) 
= H^1(\cI_\fq, T \otimes_{\bZ_p} \bQ_p)$ is injective. 
\end{example}

\begin{definition}
Let $\cF$ be a Selmer structure on $T$ and let $\fa$, $\fb$, and $\fc$
be pairwise relatively prime integral ideals of $K$. 
Define a Selmer structure $\cF^{\fa}_{\fb}(\fc)$ on $T$ by the following data:
\begin{itemize}
\item $\Sigma(\cF^{\fa}_{\fb}(\fc)) = \Sigma(\cF) \cup \{ \fq \mid \fa\fb\fc \}$,
\item define 
$H^{1}_{\cF_{\fb}^{\fa}(\fc)}(K_{\fq},T)= 
\begin{cases}
H^{1}(K_{\fq},T) & \text{if $\fq \mid \fa$}, 
\\
0 & \text{if $\fq \mid \fb$}, 
\\
H^1_{\rm tr}(K_\fq, T) & \text{if $\fq \mid \fc$}, 
\\
H^{1}_{\cF}(K_{\fq},T) & \text{otherwise}. 
\end{cases} $
\end{itemize}
\end{definition}

\begin{definition}
Let $\mu_{p^{\infty}}$ be the group of all $p$-power roots of unity. 
The Cartier dual of $T$ is defined by $T^{*}(1) := \Hom(T,\mu_{p^{\infty}})$. 
We have the local Tate paring
\begin{align*}
\langle \ ,\  \rangle_{\fq} \colon H^{1}(K_{\fq},T) \times H^{1}(K_{\fq},T^{*}(1)) \to \bQ_{p}/\bZ_{p}. 
\end{align*}
for each prime $\fq$ of $K$. 
Let $\cF$ be a Selmer structure on $T$. 
Put 
\begin{align*}
H^{1}_{\cF^{*}}(K_{\fq}, T^{*}(1)) 
:= \{ x \in H^{1}(K_{\fq}, T^{*}(1)) \mid \langle y ,x  \rangle_{\fq} = 0 
\text{ for any } y \in H^{1}_{\cF}(K_{\fq},T)\}. 
\end{align*}
Then the Selmer structure $\cF$ on $T$ defines a Selmer structure $\cF^*$ on $T^*(1)$. 
We define the dual Selmer group $H^{1}_{\cF^{*}}(K, T^{*}(1))$ associated with $\cF$ 
to be the kernel of the sum of localization maps
\begin{align*}
H^{1}_{\cF^{*}}(K, T^{*}(1)) := \ker \left( H^{1}(K, T^{*}(1)) 
\to \bigoplus_{\fq} H^{1}_{/\cF^*}(K_{\fq}, T^{*}(1)) \right) 
\end{align*}
where $\fq$ runs through all the primes of $K$. 
\end{definition}

\begin{remark}\label{dual}
If $R$ is a zero dimensional Gorenstein local ring, 
then there is a functorial isomorphism $M^* = \Hom_R(M, R) \simeq \Hom(M, \bQ_p/\bZ_p)$ as $R$-modules 
for any finitely generated $R$-module $M$. 
In fact, there exists a non-canonical isomorphism $R \simeq \Hom(R, \bQ_p/\bZ_p)$ as $R$-modules 
since 
\begin{align*}
L_R(\Hom(R, \bQ_p/\bZ_p)) = L_R(R) 
\end{align*}
and 
\begin{align*}
L_R(\Hom(R, \bQ_p/\bZ_p)[\fm_R]) = L_R(\Hom(\Bbbk, \bQ_p/\bZ_p)) = 1. 
\end{align*}
Then the image of $1$ under the isomorphism $R \simeq \Hom(R, \bQ_p/\bZ_p)$ is an injective map, 
and it induces a functorial isomorphism $M^* \simeq \Hom(M, \bQ_p/\bZ_p)$ 
since $L_R(M^*) = L_R(\Hom(M, \bQ_p/\bZ_p)) < \infty$. 
In particular, we have an $R[G_K]$-isomorphism
\begin{align*}
T^*(1) = \Hom(T, \bQ_p/\bZ_p) \otimes_{\bZ_p} \bZ_p(1) \simeq T^* \otimes_{\bZ_p} \bZ_p(1)
\end{align*}
and $T^*(1)$ is a free $R$-module of finite rank. Here $\bZ_p(1)$ is the Tate twist of $\bZ_p$. 
\end{remark}

\begin{theorem}[Poitou-Tate global duality]\label{pt}
Let $\cF_{1}$ and $\cF_2$ be Selmer structures on $T$. 
If $R$ is a zero dimensional Gorenstein local ring and 
$H^{1}_{\cF_{1}}(K_{\fq}, T) \subseteq H^{1}_{\cF_{2}}(K_\fq, T)$ for any prime $\fq$ of $K$, 
then we have the following exact sequence: 
\begin{align*}
0 \to H^{1}_{\cF_{1}}(K, T) \to H^{1}_{\cF_{2}}(K, T) 
\to \bigoplus_{\fq} H^{1}_{\cF_{2}}(K_{\fq}, T)/H^{1}_{\cF_{1}}(K_{\fq}, T) 
\to H^{1}_{\cF_{1}^{*}}(K, T^{*}(1))^* \to H^{1}_{\cF_{2}^{*}}(K, T^{*}(1))^* \to 0
\end{align*}
where $\fq$ runs through all the primes of $K$ which satisfies 
$H_{\cF_{1}}^{1}(K_{\fq}, T) \neq H^{1}_{\cF_{2}}(K_{\fq}, T)$. 
\end{theorem}
\begin{proof}
This theorem follows from \cite[Theorem 2.3.4]{MR1} and Remark \ref{dual}. 
\end{proof}

\begin{definition}[Cartesian Condition]
Suppose that $R$ is a zero dimensional Gorenstein local ring. 
We fix an injective map $\Bbbk \to R$. It induces an injective map $T/\fm_R T \to T$. 
We say that a Selmer structure $\cF$ on $T$ is cartesian if 
the map 
\begin{align*}
H^1_{/\cF}(K_\fq, T/ \fm_R T) \to H^1_{/\cF}(K_\fq, T)
\end{align*}
induced by the map $T/ \fm_R T \to T$ is injective for any prime $\fq \in \Sigma(\cF)$. 
\end{definition}

\begin{remark}\label{gorencar}
Suppose that $R$ is a zero dimensional Gorenstein local ring. 
By the definition of Gorenstein ring, the ideal $R[\fm_R]$ is principal, 
and the number of elements of the set 
$\left( \Hom_R(\Bbbk, R) \setminus \{ 0 \} \right) /R^{\times}$ is one. 
Hence the definition of cartesian condition is independent of the choice of the injective map $\Bbbk \to R$. 
\end{remark}

Let $R$ and $S$ be zero dimensional Gorenstein local rings and $\pi \colon R \to S$ a surjective ring homomorphism. 
Put $I = \ker(\pi)$. Since $R$ is an injective $R$-module, $R[I] \simeq \Hom_R(S, R)$ is an injective $S$-module. 
Since $S$ is a zero dimensional Gorenstein local ring, we conclude that there is an isomorphism 
$S \simeq R[I]$ as $R$-modules. 
In particular, there is an injective $R$-module homomorphism $S \to R$. 

\begin{lemma}\label{cartesiansurj}
Let $R$ and $S$ be zero dimensional Gorenstein local rings and $\pi \colon R \to S$ a surjective ring homomorphism. 
If a Selmer structure $\cF$ on $T$ is cartesian, then so is $\cF_S$. 
\end{lemma}
\begin{proof}
Let $S \to R$ and $\Bbbk \to S$ be injective $R$-module homomorphisms and $\fq \in \Sigma(\cF)$. 
Then these maps induce maps $H^1_{/\cF_S}(K_\fq, T \otimes_R S) \to H^1_{/\cF}(K_\fq, T)$ and 
$H^1_{/\cF}(K_\fq, T/ \fm_R T) \to H^1_{/\cF_S}(K_\fq, T \otimes_R S)$. 
By the definition of cartesinan condition and Remark \ref{gorencar}, the composition of the maps 
\begin{align*}
H^1_{/\cF_S}(K_\fq, T/ \fm_R T) = H^1_{/\cF}(K_\fq, T/ \fm_R T) \to 
H^1_{/\cF_S}(K_\fq, T \otimes_R S) \to H^1_{/\cF}(K_\fq, T)
\end{align*}
is injective, and thus so is $H^1_{/\cF_S}(K_\fq, T/ \fm_R T) \to H^1_{/\cF_S}(K_\fq, T \otimes_R S)$. 
\end{proof}

Let $\cH$ be the Hilbert class field of $K$ and $\cO_K$ the ring of integers of $K$. 
Set $\cH_{\infty} := \cH(\mu_{p^{\infty}}, (\cO_{K}^{\times})^{p^{-\infty}})$. 
In order to use the results of \cite{MR1} and \cite{MR2}, we will usually assume the following additional conditions. 
\begin{hypothesis}\label{standardhyp}
\begin{itemize}
\item[(H.1)] $(T/\fm_R T)^{G_{K}} = (T^{*}(1)[\fm_R])^{G_{K}} = 0$ and 
$T/\fm_R T$ is an irreducible $\Bbbk[G_{K}]$-module, 
\item[(H.2)] there is a $\tau \in {\rm Gal}(\overline{K}/\cH_{\infty})$ 
such that $T/(\tau - 1)T \simeq R$ as $R$-modules,  
\item[(H.3)] $H^{1}(\cH_{\infty}(T)/K, T/\fm_R T) = H^{1}(\cH_{\infty}(T)/K, T^{*}(1)[\fm_R])=0$ where 
$\cH_{\infty}(T)$ is the fixed field of the kernel of the map 
$\Gal(\overline{K}/\cH_{\infty}) \to {\rm Aut}(T)$. 
\end{itemize}
\end{hypothesis}

\begin{lemma}\label{mult}
Suppose that $R$ is a zero dimensional Gorenstein local ring and $(T/ \fm_R T)^{G_K} = 0$. 
Let $\cF$ be a cartesian Selmer structure on $T$. 
Then an injective map $\Bbbk \to R$ induces isomorphisms
\begin{itemize}
\item[(1)] $H^{1}(K, T/ \fm_R T) \simeq H^{1}(K, T)[\fm_R]$, 
\item[(2)] $H^{1}_{\cF}(K, T/\fm_R T) \simeq H^{1}_{\cF}(K, T)[\fm_R]$. 
\end{itemize}
\end{lemma}
\begin{proof}
We fix an injective map $\Bbbk \to R$. Let $x$ be the image of $1$ under the fixed map $\Bbbk \to R$ 
and $m_1, \ldots, m_d$ generators of the maximal ideal $\fm_R$ of $R$. 
Then we have exact sequences $0 \to T \otimes_R \Bbbk \to T \to T/xT \to 0$ and 
$0 \to T/xT \xrightarrow{\times (m_1, \ldots, m_d)} T^{d}$. 
Since $(T/\fm_R T)^{G_K} = 0$, these exact sequences induce exact sequences 
\begin{align*}
0 \to H^1(K, T/\fm_R T) \to H^1(K, T) \to H^1(K, T/xT)
\end{align*}
and 
\begin{align*}
0 \to H^1(K, T/xT) \to H^1(K, T^{d}) = H^1(K, T)^d. 
\end{align*}
Since the composition of the maps $H^1(K, T) \to H^1(K, T/xT) \to H^1(K, T)^d$ 
is multiplication by $(m_1, \ldots, m_d)$, we have 
\begin{align*}
H^1(K, T/\fm_R T) &= \ker\left( H^1(K, T) \to H^1(K, T/xT) \right)
\\
&= \ker\left( H^1(K, T) \xrightarrow{\times (m_1, \ldots, m_d)} H^1(K, T)^d \right)
\\
&= H^1(K, T)[\fm_R]. 
\end{align*} 
To prove the assertion (2), we consider the following exact sequences:
\begin{align*}
\xymatrix{
0 \ar[r] & H^1_{\cF}(K, T/\fm_R T) \ar[r] \ar[d] & H^1(K, T/\fm_R T) \ar[r] \ar[d] & 
\bigoplus_\fq H^1_{/\cF}(K_\fq, T/\fm_R T) \ar[d]
\\
0 \ar[r] & H^1_{\cF}(K, T)[\fm_R] \ar[r] & H^1(K, T)[\fm_R] \ar[r] & \bigoplus_\fq H^1_{/\cF}(K_\fq, T)
}
\end{align*}
where $\fq$ runs through all the primes of $K$ and 
the vertical maps are induced by the fixed map $\Bbbk \to R$. 
Since the middle vertical map is an isomorphism, it suffices to show that the map 
$H^1_{/\cF}(K_\fq, T/\fm_R T) \to H^1_{/\cF}(K_\fq, T)$ is injective for any prime $\fq$ of $K$. 
If $\fq \in \Sigma(\cF)$, this map is injective by the definition of cartesian condition. 
Let $\fq \not\in \Sigma(\cF)$ and $M \in \{T, T/\fm_R T\}$. 
Then $M$ is unramified at $\fq$ and $H^1_{\cF}(K_\fq, M) = H^1_f(K_\fq, M)$.  
Since the cohomological dimension of $\cD_\fq/\cI_\fq \simeq \widehat{\bZ}$ is one, 
by the inflation-restriction exact sequence, we have 
\begin{align*}
H^1_{/\cF}(K_\fq, M) \simeq H^1(\cI_\fq, M)^{{\rm Fr}_\fq = 1} = \Hom(\cI_\fq, M)^{{\rm Fr}_\fq = 1}. 
\end{align*}
Hence the map $H^1_{/\cF}(K_\fq, T/\fm_R T) \to H^1_{/\cF}(K_\fq, T)$ is injective. 
\end{proof}

\begin{lemma}\label{dmult}
Let $\cF$ be a Selmer structure on $T$ and $I$ an ideal of $R$. 
Suppose that $R$ is an artinian local ring and $(T^*(1)[\fm_R])^{G_K} = 0$. 
Then the inclusion map $T^*(1)[I] \to T^*(1)$ induces isomorphisms
\begin{itemize}
\item[(1)] $H^{1}(K, T^*(1)[I]) \simeq H^{1}(K, T^*(1))[I]$, 
\item[(2)] $H^{1}_{\cF^*}(K, T^*(1)[I]) \simeq H^{1}_{\cF^*}(K, T^*(1))[I]$. 
\end{itemize}
\end{lemma}
\begin{proof}
The proof of this lemma is same as that of \cite[Lemma 3.5.3]{MR1}. 
\end{proof}

\begin{definition} 
Let $\cF$ be a Selmer structure on $T$. 
Define a set $\cP(\cF)$ of primes of $K$ by $\cP(\cF) = \{ \fq \mid \fq \not\in \Sigma(\cF)\}$. 
Let $\tau$ be as in (H.2) and $q$ the order of the residue field $\Bbbk$ of $R$. 
Set $\cH_n = \cH(\mu_{q^n},(\cO_{K}^{\times})^{q^{-n}})$ for each positive integer $n$. 
We define a set $\cP_n(\cF)$ of primes of $K$ as 
\begin{align*}
\cP_n(\cF) = \{ \fq \in \cP(\cF) \mid \fq &\text{ is unramified in $\cH_n(T/\fm_R^n T)/K$}, 
\text{ ${\rm Fr}_\fq$ is conjugate to $\tau$ in $\Gal(\cH_n(T/\fm_R^n T)/K)$}  \}. 
\end{align*}
\end{definition}

\begin{remark}
If $\pi \colon R \to S$ is a surjective ring homomorphism, 
then we have $\cP_n(\cF) \subseteq \cP_n(\cF_S)$ for any positive integer $n$. 
\end{remark}

Let $\fq$ be a prime of $K$ where $T$ is unramified. 
We define the singular quotient at $\fq$ as 
\begin{align*}
H^1_{/f}(K_{\fq}, T) := H^{1}(K_{\fq}, T)/H^1_{f}(K_\fq, T). 
\end{align*}
We denote by ${\rm loc}^{/f}_\fq \colon H^1(K, T) \to H^1_{/f}(K_\fq, T)$ 
the composition of the localization map at $\fq$ 
and a surjective map $H^1(K_\fq, T) \to H^1_{/f}(K_\fq, T)$. 

Let $\cN (\cP)$ denote the set of square-free products of primes in a set $\cP$ of primes of $K$. 
By considering empty product, the trivial ideal $1$ is contained in $\cN(\cP)$. 

\begin{lemma}\label{local}
Suppose that $R$ is an artinian ring of length $n > 0$ and $T$ satisfies (H.2). 
If $\fq \in \cP_n(\cF)$, 
then both $H^1_f(K_{\fq}, T)$ and $H^1_{/f}(K_\fq, T)$ are free $R$-modules of rank one 
and the composition of maps
\begin{align*}
H^1_{\rm tr}(K_\fq, T) \to H^1(K_\fq, T) \to H^1_{/f}(K_\fq, T)
\end{align*}
is an isomorphism. 
Furthermore, we have canonical isomorphisms 
\begin{itemize}
\item[(1)] $H^1_f(K_\fq, T) \otimes_R R/I \simeq H^1_f(K_\fq, T/IT)$, 
\item[(2)] $H^1_{/f}(K_\fq, T) \otimes_R R/I \simeq H^1_{/f}(K_\fq, T/IT)$, 
\end{itemize}
for any ideal $I$ of $R$. 
In particular, the map $H^1(K_\fq, T) \otimes_R R/I \to H^1(K_\fq, T/IT)$ is an isomorphism. 
\end{lemma}
\begin{proof}
This lemma follows from \cite[Lemma 1.2.3]{MR1} and \cite[Lemma 1.2.4]{MR1}. 
\end{proof}

\begin{corollary}\label{cartesian}
Suppose that $R$ is a zero dimensional Gorenstein local ring of length $n > 0$ and $T$ satisfies (H.2). 
Let $\fa, \fb, \fc \in \cN(\cP_n(\cF))$ with $\fa\fb\fc \in \cN(\cP_n(\cF))$. 
If a Selmer structure $\cF$ on $T$ is cartesian, then so is $\cF^\fa_\fb(\fc)$. 
\end{corollary}
\begin{proof}
We fix an injective map $i \colon \Bbbk \to R$. Let $\fq \mid \fa\fb\fc$ be a prime of $K$. 
We only need to show that the map 
\begin{align*}
H^1_{/\cF^\fa_\fb(\fc)}(K_\fq, T/ \fm_R T) \to H^1_{/\cF^\fa_\fb(\fc)}(K_\fq, T)
\end{align*}
induced by the map $i \colon \Bbbk \to R$ is injective. 
By the definition of the Selmer structure $\cF^\fa_\fb(\fc)$ and Lemma \ref{local}, 
the map $H^1(K_\fq, T) \to H^1(K_\fq, T/\fm_R T)$ induces an isomorphism 
$H^1_{/\cF^\fa_\fb(\fc)}(K_\fq, T) \otimes_R \Bbbk \simeq H^1_{/\cF^\fa_\fb(\fc)}(K_\fq, T/ \fm_R T)$. 
Since $H^1_{/\cF^\fa_\fb(\fc)}(K_\fq, T)$ is a free $R$-module by Lemma \ref{local} and 
the composition of the maps 
\begin{align*}
H^1_{/\cF^\fa_\fb(\fc)}(K_\fq, T) \otimes_R \Bbbk \simeq H^1_{/\cF^\fa_\fb(\fc)}(K_\fq, T/ \fm_R T) 
\to H^1_{/\cF^\fa_\fb(\fc)}(K_\fq, T) 
\end{align*}
is ${\rm id}_{H^1_{/\cF^\fa_\fb(\fc)}(K_\fq, T)} \otimes i$, 
the map $H^1_{/\cF^\fa_\fb(\fc)}(K_\fq, T/ \fm_R T) \to H^1_{/\cF^\fa_\fb(\fc)}(K_\fq, T)$ is injective. 
\end{proof}

\begin{definition}\label{core rank}
We define the core rank $\chi(\cF)$ of a Selmer structure $\cF$ on $T$ by 
\begin{align*}
\chi(\cF) := \dim_{\Bbbk}H^{1}_{\cF}(K, T/\fm_R T) - \dim_{\Bbbk}H^{1}_{\cF^{*}}(K, T^{*}(1)[\fm_R]). 
\end{align*}
\end{definition}

\begin{example}\label{cancore}
Let $R$ be the ring of integers of a finite extension of the field $\bQ_p$. 
Suppose that $(T/\fm_R T)^{G_K} = (T^*(1)[\fm_R])^{G_K} = 0$. 
Then by the same proof of \cite[Theorem 5.2.15]{MR1}, we have 
\begin{align*}
\chi(\cF_{\rm can}) = \sum_{v \mid \infty} {\rm corank}_{R} \left( H^{0}(K_{v}, T^{*}(1)) \right) + 
\sum_{\fp \mid p} {\rm rank}_R \left( H^2(K_\fq, T) \right) 
\end{align*}
where $\cF_{\rm can}$ is the canonical Selmer structure defined in Example \ref{cansel} and 
\begin{align*}
{\rm corank}_R(M) := {\rm rank}_R (\Hom(M, \bQ_p/\bZ_p))
\end{align*}
for any finitely generated $R$-module $M$. 
\end{example}

\begin{corollary}\label{corerank}
Suppose that $R$ is an artinian local ring of length $n > 0$ and $T$ satisfies (H.2). 
Let $\cF$ be a Selmer structure on $T$ and $\fa, \fb, \fc \in \cN(\cP_n(\cF))$ with 
$\fa\fb\fc \in \cN(\cP_n(\cF))$. 
Then we have 
\begin{align*}
\chi(\cF^\fa_\fb(\fc)) = \chi(\cF) + \nu(\fa) - \nu(\fb)
\end{align*}
where $\nu(\fn)$ denotes the number of prime factors of $\fn \in \cN(\cP_n(\cF))$. 
\end{corollary}
\begin{proof}
Applying Theorem \ref{pt} with $\cF_1 = \cF$ and $\cF_2 = \cF^{\fa\fc}$, 
we have an exact sequence 
\begin{align*}
0 \to H^1_{\cF}(K, T/\fm_R T) \to &H^1_{\cF^{\fa\fc}}(K, T/\fm_R T) \to 
\bigoplus_{\fq \mid \fa\fc} H^1_{/f}(K_\fq, T/\fm_R T) 
\\
&\to H^1_{\cF^*}(K, T^*(1)[\fm_R])^* \to H^1_{(\cF^{\fa\fc})^*}(K, T^*(1)[\fm_R])^* \to 0. 
\end{align*}
By the definition of core rank and Lemma \ref{local}, we have 
\begin{align*}
\chi(\cF^{\fa\fc}) &= \chi(\cF) + \sum_{\fq \mid \fa\fc} \dim_\Bbbk H^1_{/f}(K_\fq, T/\fm_R T) 
= \chi(\cF) + \nu(\fa) + \nu(\fc). 
\end{align*}
Again using Lemma \ref{local} and Theorem \ref{pt} with $\cF_1 = \cF^\fa_\fb(\fc)$ and $\cF_2 = \cF^{\fa\fc}$, 
we have $\chi(\cF^\fa_\fb(\fc)) = \chi(\cF^{\fa\fc}) - \nu(\fb) - \nu(\fc) = \chi(\cF) + \nu(\fa) - \nu(\fb)$. 
\end{proof}

\begin{proposition}\label{chevotarev}
Suppose that $R$ is an artinian local ring and $T$ satisfies Hypothesis \ref{standardhyp}. 
Let $c \in H^{1}(K, T)$ and $c^* \in H^{1}(K, T^*(1))$ be non-zero elements and $n$ a positive integer. 
Then there is a set $\cQ \subseteq \cP_n(\cF)$ of positive density such that 
${\rm loc}_\fq(c) \neq 0$ and ${\rm loc}_\fq(c^*) \neq 0$ for every $\fq \in \cQ$.
\end{proposition}
\begin{proof}
The proof of this proposition is same as the proof of \cite[Proposition 3.6.1]{MR1}. 
\end{proof}

\begin{remark}
Proposition \ref{chevotarev} is a weaker version of \cite[Proposition 3.6.1]{MR1}. 
Thus we do not need the assumption (H.4) defined in \cite{MR1} or \cite{MR2} 
to prove Proposition \ref{chevotarev}. 
\end{remark}

\begin{lemma}\label{easy}
Suppose that $R$ is an artinian local ring and $T$ satisfies Hypothesis \ref{standardhyp}. 
Let $c \in H^1(K, T)$ with $Rc \simeq R$ as $R$-modules and $n$ a positive integer. 
Then there are infinitely many primes $\fq \in \cP_n(\cF)$ 
such that $R \cdot {\rm loc}_\fq(c) = H^1_f(K_\fq, T)$. 
\end{lemma}
\begin{proof}
We may assume that $n > L_R(R)$. 
Note that ${\rm loc}_\fq(c) \in H_{f}^1(K_\fq, T)$ for all but finitely many primes $\fq$ of $K$ and 
that $H^1_f(K_\fq, T)$ is free of rank one for any prime $\fq \in \cP_n (\cF)$ by Lemma \ref{local}. 
Since $R$ is an artinian local ring, there is an element $x \in R$ such that $\fm_R = {\rm Ann}_{R}(x)$. 
Put $c' = xc$. 
Since $Rc \simeq R$ as $R$-modules, the element $c'$ is non-zero. 
Hence by Proposition \ref{chevotarev}, there are infinitely many primes $\fq \in \cP_n(\cF)$ 
such that ${\rm loc}_\fq (c') \neq 0$ and ${\rm loc}_\fq (c) \in H^1_f (K_\fq, T)$. 
Since $\fm_R = {\rm Ann}_R(x)$, 
and $H^1_{f}(K_\fq, T) \simeq R$ as $R$-modules, we have $R \cdot {\rm loc}_\fq(c) = H^1_f(K_\fq, T)$. 
\end{proof}

\begin{lemma}\label{basis}
Suppose that $R$ is an artinian local ring and $T$ satisfies Hypothesis \ref{standardhyp}. 
Let $M$ be a free $R$-submodule of $H^{1}(K, T)$ of rank $s \geq 0$ and $R \to A$ a ring homomorphism.  
Then the composition of the maps 
\begin{align*}
M \otimes_R A \to H^{1}(K, T) \otimes_R A \to H^{1}(K, T \otimes_R A)
\end{align*}
is a split injection. 
Here $T \otimes_R A$ is equipped with the discrete topology. 
\end{lemma}
\begin{proof}
Set $n = L_R(R)$. 
By Lemma \ref{easy}, there are primes $\fq_{1}, \cdots, \fq_{s} \in \cP_n(\cF)$ such that 
the sum of localization maps 
$\bigoplus_{i=1}^{s}{\rm loc}_{\fq_{i}} \colon H^1(K, T) \to \bigoplus_{i=1}^{s}H^{1}(K_{\fq_i}, T)$ 
induces an isomorphism $M \simeq \bigoplus_{i=1}^{s}H^{1}_{f}(K_{\fq}, T)$. 
Let $1 \leq i \leq s$. 
Since $T$ is unramified at $\fq_i$, so is $T \otimes_R A$, which implies that 
\begin{align*}
H^{1}_{f}(K_{\fq_i}, T) \otimes_R A &\simeq \left( T/({\rm Fr}_{\fq_i}-1)T \right) \otimes_R A 
\\
&\simeq (T \otimes_R A)/ \left( ({\rm Fr}_{\fq_i} - 1)(T \otimes_R A) \right)
\\
&\simeq H^1_f(K_{\fq_i}, T \otimes_R A). 
\end{align*}
Therefore the map $M \otimes_R A \to \bigoplus_{i=1}^{s}H^{1}_{f}(K_{\fq}, T \otimes_R A)$ 
is an isomorphism. 
\end{proof}

\section{Stark Systems over Zero Dimensional Gorenstein Local Rings}

In this section, we define a Stark system over a zero dimensional Gorenstein local ring. 
Under some mild assumptions, 
we will prove that the module of Stark systems is free of rank one and 
control all the Fitting ideals of the Pontryagin dual of dual Selmer groups using Stark systems. 
Furthermore, we will show that these results generalize the works of B. Mazur and K. Rubin 
in \cite{MR1} and \cite{MR2}. 

Throughtout this section, we assume that $R$ is a zero dimensional Gorenstrein local ring and 
that $\cF$ is a Selmer structure on $T$. 
Suppose that $T$ satisfies Hypothesis \ref{standardhyp}. 

For simplicity, we write $\cP = \cP_{L_R(R)}(\cF)$ and $\cN = \cN(\cP_{L_R(R)}(\cF))$. 
Let $\nu(\fn)$ denote the number of prime factors of $\fn \in \cN$. 

\subsection{Stark Systems over Zero Dimensional Gorenstein Local Rings}

Let $r$ be a non-negative integer and $\fn \in \cN$. Define 
$W_{\fn} :=  \bigoplus_{\fq \mid \fn} H^{1}_{/f}(K_{\fq}, T)^* 
= \bigoplus_{\fq \mid \fn} \Hom_R \left( H^{1}_{/f}(K_{\fq}, T), R \right)$ and 
\begin{align*} 
X_{\fn}(T, \cF) := \bigcap^{r + \nu(\fn)}_R H^{1}_{\cF^{\fn}}(K, T) \otimes_R 
{\rm det}(W_{\fn}). 
\end{align*}
If there is no risk of confusion, we abbreviate $X_{\fn}(T, \cF)$ to $X_{\fn}$. 
Let $\fm \mid \fn$ and $1 \leq i \leq \nu(\fn\fm^{-1})$. 
Order the primes $\fq_1, \ldots, \fq_{\nu(\fn\fm^{-1})}$ dividing $\fn\fm^{-1}$. 
We have the following cartesian diagram: 
\begin{align*}
\xymatrix{
H^{1}_{\cF^{\fm \fq_1 \cdots \fq_{i-1}}}(K, T) \ar[d] \ar@{^{(}->}[r] 
& H^{1}_{\cF^{\fm \fq_1 \cdots \fq_i}}(K, T) \ar[d] 
\\
\bigoplus_{\fq \mid \fm \fq_1 \cdots \fq_{i-1}} H^{1}_{/f}(K_{\fq}, T) \ar@{^{(}->}[r] & 
\bigoplus_{\fq \mid \fm \fq_1 \cdots \fq_{i}} H^{1}_{/f}(K_{\fq}, T). 
}
\end{align*}
By Definition \ref{map}, we obtain a map 
$\Phi_{\fm \fq_1 \cdots \fq_{i}, \fm \fq_1 \cdots \fq_{i-1}} 
\colon X_{\fm \fq_1 \cdots \fq_i} \to X_{\fm \fq_1 \cdots \fq_{i-1}}$ 
for each $1 \leq i \leq \nu(\fn\fm^{-1})$. 
Then we define a map 
\begin{align*}
\Phi_{\fn, \fm} \colon X_{\fn} \to X_\fm
\end{align*} 
by $\Phi_{\fn, \fm} = \Phi_{\fm\fq_1, \fm} \circ \cdots \circ \Phi_{\fn, \fn\fq_{\nu(\fn\fm^{-1})}^{-1}}$. 

\begin{lemma}\label{compatible}
The map $\Phi_{\fn, \fm}$ is independent of the choice of the ordering 
$\fn\fm^{-1} = \fq_1 \cdots \fq_{\nu(\fn\fm^{-1})}$. 
Furthermore, we have $\Phi_{\fn_{1}, \fn_{3}} = \Phi_{\fn_{2}, \fn_{3}} \circ \Phi_{\fn_{1}, \fn_{2}}$ 
for any $\fn_{1} \in \cN$, $\fn_{2} \mid \fn_{1}$, and $\fn_{3} \mid \fn_{2}$. 
\end{lemma}
\begin{proof}
This lemma follows from Proposition \ref{homlem}. 
\end{proof}

\begin{remark}\label{sign}
Let $\fn \in \cN$ and $\fm \mid \fn$. 
Applying Lemma \ref{base} to the exact sequence 
\begin{align*}
W_{\fn\fm^{-1}} \to H^1_{\cF^\fn}(K, T)^* \to H^1_{\cF^\fm}(K, T)^* \to 0, 
\end{align*}
we get a map $\widetilde{\Phi}_{\fn, \fm} \colon 
\bigcap^{r + \nu(\fn)}_R H^1_{\cF^\fn}(K, T) \otimes_R \det(W_{\fn\fm^{-1}}) \to 
\bigcap^{r + \nu(\fm)}_R H^1_{\cF^\fm}(K, T)$. 
By Proposition \ref{homlem}, we see that 
\begin{align*}
\Phi_{\fn, \fm} = 
(-1)^{\mu(\fn\fm^{-1})} 
\widetilde{\Phi}_{\fn, \fm} \otimes {\rm id}_{\det(W_{\fm})} 
\end{align*}
where $\mu(\fn\fm^{-1}) = \sum_{i=1}^{\nu(\fn\fm^{-1})-1}i$. 
\end{remark}

\begin{definition}
Let $\cQ$ be a subset of $\cP$. 
We define the module ${\bf SS}_{r}(T, \cF, \cQ)$ of Stark systems of rank $r$ for $(T, \cF, \cQ)$ by 
\begin{align*}
{\bf SS}_{r}(T, \cF, \cQ) := \varprojlim_{\fn \in \cN(\cQ)} X_{\fn}
\end{align*}
where the inverse limit is taken with respect to the maps $\Phi_{\fn, \fm}$. 
We call an element of ${\bf SS}_{r}(T, \cF, \cQ)$ a Stark system of rank $r$ 
associated with the tuple $(T, \cF, \cQ)$. 
\end{definition}

\begin{definition}
We say that $\fn \in \cN$ is a core vertex (for $\cF$) if 
$H^{1}_{(\cF^{\fn})^{*}}(K, T^{*}(1)) = H^{1}_{\cF_{\fn}^{*}}(K, T^{*}(1)) = 0$. 
\end{definition}

\begin{remark}\label{corerem}
\begin{itemize}
\item[(1)] By Lemma \ref{chevotarev}, 
the set $\cN(\cP_n(\cF))$ has a core vertex for any positive integer $n$. 
\item[(2)] If $\fq \in \cP$ and $\fn \in \cN(\cP \setminus \{ \fq \})$ is a core vertex, 
then $\fn\fq \in \cN$ is also a core vertex. 
\item[(3)] Let $\fn \in \cN$ be a core vertex, $S$ a zero dimensional Gorenstein local ring, 
and $\pi \colon R \to S$ a surjective ring homomorphism. 
Then $\fn$ is also a core vertex for $\cF_S$ by Lemma \ref{dmult}. 
\end{itemize}
\end{remark}

\begin{lemma}\label{freelemma}
Suppose that the Selmer structure $\cF$ on $T$ is cartesian. 
If $\fn \in \cN$ is a core vertex, 
then the $R$-module $H^1_{\cF^\fn}(K, T)$ is free of rank $\chi(\cF) + \nu(\fn)$. 
\end{lemma}
\begin{proof}
Let $\fn \in \cN$ be a core vertex. 
We take an ideal $\fa \in \cN$ such that $(\fa, \fn) = 1$ and $\nu(\fa) = \chi(\cF)$. 
By Corollary \ref{corerank}, we have $\chi(\cF_\fa) = \chi(\cF) - \nu(\fa) = 0$. 
Then by \cite[Theorem 11.6]{MR2}, there is an ideal $\fm \in \cN$ such that $(\fa\fn, \fm) = 1$ and 
$H^1_{\cF_\fa(\fm)^*}(K, T^*(1)[\fm_R]) = 0$. 
Then we have 
\begin{align*}
\dim_\Bbbk H^1_{\cF_\fa(\fm)}(K, T/\fm_R T) = \chi(\cF_\fa(\fm)) = \chi(\cF_\fa) = 0 
\end{align*}
where the first equality follows from $H^1_{\cF_\fa(\fm)^*}(K, T^*(1)[\fm_R]) = 0$ and 
the second equality follows from Corollary \ref{corerank}. 
Since $\cF$ is cartesian, it follows from Lemma \ref{mult}, Lemma \ref{dmult}, and Corollary \ref{cartesian} that 
\begin{align*}
H^1_{\cF_\fa(\fm)}(K, T) = H^1_{\cF_\fa(\fm)^*}(K, T^*(1)) = 0. 
\end{align*}
Applying Theorem \ref{pt} with $\cF_1 = \cF_\fa(\fm)$ and $\cF_2 = \cF^{\fm\fn}$, we see that 
the canonical map 
\begin{align*}
H^1_{\cF^{\fm\fn}}(K, T) \to \bigoplus_{\fq \mid \fa} H^1_f(K_\fq, T) \oplus 
\bigoplus_{\fq \mid \fm}H^1(K_\fq, T)/H^1_{\rm tr}(K_\fq, T) 
\oplus \bigoplus_{\fq \mid \fn}H^1_{/f}(K_\fq, T)
\end{align*}
is an isomorphism. Hence by Lemma \ref{local}, $H^1_{\cF^{\fm\fn}}(K, T)$ is a free $R$-module of rank 
$\chi(\cF) + \nu(\fm) + \nu(\fn)$. 
Since $\fn$ is a core vertex, by Theorem \ref{pt}, we have an exact sequence 
\begin{align*}
0 \to H^1_{\cF^\fn}(K, T) \to H^1_{\cF^{\fm\fn}}(K, T) \to \bigoplus_{\fq \mid \fm} H^1_{/f}(K_\fq, T) \to 0. 
\end{align*}
Again by Lemma \ref{local}, $\bigoplus_{\fq \mid \fm} H^1_{/f}(K, T)$ is a free $R$-module of rank $\nu(\fm)$, 
which implies that 
$H^1_{\cF^\fn}(K, T)$ is free of rank $\chi(\cF) + \nu(\fn)$. 
\end{proof}

\begin{theorem}\label{core}
Suppose that $\cF$ is cartesian and $r = \chi(\cF) \geq 0$. 
Let $\cQ$ be a subset of $\cP$. 
If the set $\cN(\cQ)$ has a core vertex, then the projection map 
\begin{align*}
{\bf SS}_r(T, \cF, \cQ) \to X_{\fn}(T, \cF)
\end{align*}
is an isomorphism for any core vertex $\fn \in \cN(\cQ)$. 
In particular, the module ${\bf SS}_{r}(T, \cF, \cQ)$ is a free $R$-module of rank one. 
\end{theorem}
\begin{proof}
Let $\fn \in \cN(\cQ)$ be a core vertex. 
We only need to show that 
the map $\Phi_{\fn\fq, \fn}$ is an isomorphism for any prime $\fq \in \cQ$ with $\fq \nmid \fn$. 
Note that $\fn\fq$ is also a core vertex. 
By Theorem \ref{pt} and Lemma \ref{freelemma}, we have the following exact sequence of free $R$-modules: 
\begin{align*}
0 \to H^{1}_{/f}(K_{\fq}, T)^* \to H^{1}_{\cF^{\fn\fq}}(K, T)^* \to H^{1}_{\cF^{\fn}}(K, T)^* \to 0.  
\end{align*}
Since ${\rm rank}_R \left( H^1_{\cF^{\fn\fq}}(K, T) \right) = r + \nu(\fn\fq)$ 
and ${\rm rank}_R \left( H^1_{\cF^\fn}(K, T) \right) = r + \nu(\fn)$, 
the map 
\begin{align*}
\bigwedge^{r+\nu(\fn)}_R H^1_{\cF^{\fn\fq}}(K,T)^* \otimes_R H^1_{/f}(K_\fq, T)^* \to \bigwedge^{r+\nu(\fn\fq)}_R H^1_{\cF^\fn}(K,T)^*
\end{align*}
defined by Lemma \ref{base} is an isomorphism, 
and thus so is $\Phi_{\fn\fq, \fn}$. 
\end{proof}


\begin{lemma}\label{keylem}
Let $s$ and $t$ be non-negative integers and 
$0 \to N \to R^{s+t} \to R^{s} \to M \to 0$ an exact sequence of $R$-modules. 
If $\phi \in \Hom_R \left(\bigwedge^r_R N^*, R \right)$ is the image of a basis of 
$\bigcap^{s+t}_R R^{s+t} \otimes_R \det\left( (R^{s})^* \right)$ under the map 
\begin{align*}
\bigcap^{s+t}_R R^{s+t} \otimes_R \det\left( (R^{s})^* \right) \to \bigcap^{t}_R N 
= \Hom_R \left(\bigwedge^r_R N^*, R \right)
\end{align*}
defined in Definition \ref{map}, then we have ${\rm im}(\phi) = {\rm Fitt}_{0, R}(M)$. 
\end{lemma}
\begin{proof}
Applying Lemma \ref{base} to the exact sequence 
$(R^s)^* \to (R^{s+t})^* \to N^* \to 0$, we have a map 
\begin{align*}
\bigwedge^{t}_R N^* \otimes_R \det((R^s)^*) \to \bigwedge^{s+t}_R (R^{s+t})^* \simeq R 
\end{align*}
and ${\rm im}(\phi)$ is its image. 
The injective map $N \to R^{s+t}$ induces a surjective map 
$\bigwedge^{t}_R (R^{s+t})^* \otimes_R \det((R^s)^*) \to \bigwedge^{t}_R N^* \otimes_R \det((R^s)^*)$. 
Let 
\begin{align*}
\Psi \colon \bigwedge^{t}_R (R^{s+t})^* \otimes_R \det((R^s)^*) \to 
\bigwedge^{t}_R N^* \otimes_R \det((R^s)^*) \to R. 
\end{align*}
Then we have ${\rm im}(\phi) = {\rm im}(\Psi)$. 
Let $r_1, \ldots, r_{s+t}$ be the standard basis of $R^{s+t}$ and 
$r_1', \ldots, r_s'$ be the standard basis of $R^s$. 
We denote by $r_1^* \ldots, r_{s+t}^*$ (resp. $r_1'^*, \ldots, r_s'^*)$ the dual basis of 
$r_1 \ldots, r_{s+t}$ (resp. $r_1', \ldots, r_s')$. 
Let $A = (a_{ij})$ be the $s \times (s+t)$-matrix defined by the map $R^{s+t} \to R^{s}$. 
If $\{i_1 < \cdots < i_t \} \subseteq \{1, \ldots, s+t\}$, then we compute 
\begin{align*}
\Psi(r_{i_1}^* \wedge \cdots \wedge r_{i_t}^* \otimes r_1'^* \wedge \cdots \wedge r_s'^*) = 
\pm \det((a_{kj_\nu})_{k, \nu})
\end{align*}
where $\{j_{1} < \cdots < j_{s} \}$ is the complement of $\{i_1 < \cdots < i_t \}$ in $\{1, \ldots, s+t\}$. 
By the definition of Fitting ideals, we have ${\rm im}(\phi) = {\rm im}(\Psi) = {\rm Fitt}_{0, R} (M)$.  
\end{proof}

\begin{definition}
Let $\cQ$ be a subset of $\cP$ and $\epsilon \in {\bf SS}_{r}(T, \cF, \cQ)$. 
Fix an isomorphism $H^1_{/f}(K_\fq, T) \simeq R$ for each prime $\fq \in \cQ$. 
Using these isomorphisms, we have an isomorphism 
\begin{align*}
i_\fn \colon X_{\fn} \simeq \bigcap^{r+\nu(\fn)}_R H^1_{\cF^\fn}(K, T) 
= \Hom_R \left( \bigwedge^{r+\nu(\fn)}_R H^1_{\cF^\fn}(K, T)^*, R \right)
\end{align*}
for each ideal $\fn \in \cN(\cQ)$. 
Then we define an ideal $I_{i}(\epsilon) \subseteq R$ by 
\begin{align*}
I_i(\epsilon) := \sum_{\fn \in \cN(\cQ), \nu(\fn) = i} 
{\rm im}(i_\fn(\epsilon_{\fn})) 
\end{align*}
for any non-negative integer $i$. It is easy to see that the ideal $I_i(\epsilon)$ is 
independent of the choice of the isomorphisms $H^1_{/f}(K_\fq, T) \simeq R$. 
\end{definition}

\begin{theorem}\label{control}
Suppose that $\cF$ is cartesian and $r = \chi(\cF) \geq 0$. 
Let $\cQ$ be an infinite subset of $\cP$ such that $\cN(\cQ)$ has a core vertex.  
If $\epsilon = \{\epsilon_\fn\}_{\fn \in \cN(\cQ)}$ is a basis of ${\bf SS}_{r}(T, \cF, \cQ)$, 
then we have 
\begin{align*}
I_{i}(\epsilon) = {\rm Fitt}_{i, R}\left( H^{1}_{\cF^*}(K, T^*(1))^* \right)
\end{align*}
for any non-negative integer $i$. 
\end{theorem}
\begin{proof}
Let $i$ be a non-negative integer. Since $\cN(\cQ)$ has a core vertex and 
$\cQ$ is an infinite set, we can take a core vertex $\fn \in \cN(\cQ)$ with $\nu(\fn) \geq i$. 
Fix an isomorphism $H^1_{/f}(K_\fq, T) \simeq R$ for each prime $\fq \in \cQ$. 
These isomorphisms induce an isomorphism 
$i_\fm \colon X_{\fm} \simeq \bigcap^{r+\nu(\fm)}_R H^1_{\cF^\fm}(K, T)$ for each ideal $\fm \in \cN(\cQ)$. 
Since $\fn$ is a core vertex, by Theorem \ref{pt}, we have an exact sequence 
\begin{align*}
0 \to H^{1}_{\cF^{\fm}}(K, T) \to H^{1}_{\cF^\fn}(K, T) 
\to \bigoplus_{\fq \mid \fm^{-1}\fn}H^{1}_{/f}(K_{\fq}, T) 
\to H^{1}_{\cF_{\fm}^*}(K, T^*(1))^* \to 0 
\end{align*}
for an ideal $\fm \mid \fn$. 
We denote by $e(\fm)$ this exact sequence. 
Note that $H^1_{\cF^\fn}(K, T)$ is free of rank $r + \nu(\fn)$ by Lemma \ref{freelemma} and that 
$H^1_{/f}(K_\fq, T)$ is free of rank one for any prime $\fq \in \cQ$ by Lemma \ref{local}. 
Since $\fn$ is a core vertex, by Theorem \ref{core}, the element $\epsilon_\fn$ is a basis of $X_\fn$. 
Thus for any ideal $\fm \mid \fn$, the exact sequence $e(\fm)$ and 
the element $i_\fm(\epsilon_\fm)$ satisfy the assumptions in Lemma \ref{keylem}. Hence we have 
\begin{align*}
{\rm im}(i_\fm(\epsilon_\fm)) = {\rm Fitt}_{0, R}\left(H^1_{\cF_\fm^*}(K, T^*(1))^* \right). 
\end{align*}
By comparing the exact sequences $e(1)$ and $e(\fm)$ for any $\fm \mid \fn$ with $\nu(\fm) = i$, 
we see that 
\begin{align*}
{\rm Fitt}_{i, R}\left( H^{1}_{\cF^*}(K, T^*(1))^* \right) 
= \sum_{\fm \mid \fn, \nu(\fm) = i} {\rm Fitt}_{0, R}\left( H^{1}_{\cF_\fm^*}(K, T^*(1))^* \right). 
\end{align*}
Thus we have 
\begin{align*}
{\rm Fitt}_{i, R}\left( H^{1}_{\cF^*}(K, T^*(1))^* \right) 
= \sum_{\fm \mid \fn, \nu(\fm) = i} {\rm im}(i_\fm(\epsilon_\fm)) \subseteq I_i(\epsilon) 
\end{align*}
for any core vertex $\fn \in \cN(\cQ)$ with $\nu(\fn) \geq i$.  
Let $\fm \in \cQ$ with $\nu(\fm) = i$. Then there is a core vertex $\fn \in \cN(\cQ)$ with $\fm \mid \fn$. 
Thus we have 
\begin{align*}
{\rm im}(i_\fm(\epsilon_\fm)) \subseteq \sum_{\fm \mid \fn, \nu(\fn) = i} {\rm im}(i_\fm(\epsilon_\fm))
= {\rm Fitt}_{i, R}\left( H^{1}_{\cF^*}(K, T^*(1))^* \right). 
\end{align*}
Hence we get the desired equality. 
\end{proof}

The following proposition will be used in Section 6. 

\begin{proposition}\label{Fitt}
Suppose that $\cF$ is cartesian and $r = \chi(\cF) \geq 0$. 
Let $\cQ$ be a subset of $\cP$ such that $\cN(\cQ)$ has a core vertex. 
Let $\epsilon = \{\epsilon_\fn\}_{\fn \in \cN(\cQ)}$ be a basis of ${\bf SS}_{r}(T, \cF, \cQ)$. 
For $\fn \in \cN(\cQ)$, let 
$l_\fn \colon \bigwedge^{r}_R H^{1}_{\cF^\fn}(K, T) \to \bigcap^{r}_R H^{1}_{\cF^\fn}(K, T)$ denote the canonical map. 
If there is a free $R$-submodule $\sum_{1 \leq i\leq r}Rc_i$ of $H^{1}_{\cF}(K, T)$ of rank $r$, 
then there is a unique element $\theta(\epsilon) \in R$ such that 
$\theta(\epsilon) R = {\rm Fitt}_{0, R}\left(H^{1}_{\cF^*}(K, T^*(1))^* \right)$ and 
$\epsilon_{1} = l_1(\theta(\epsilon) c_{1} \wedge \cdots \wedge c_{r})$. 
In particular, we have $\epsilon_1 = 0$ if there is an injective map $R^{r+1} \to H^1_\cF(K, T)$. 
\end{proposition}
\begin{proof}
Put $M = \sum_{1 \leq i\leq r}Rc_i$. 
Note that $M$ is an injective $R$-module since $R$ is a zero dimensional Gorenstein local ring. 
Hence the map $M \to H^1_{\cF}(K, T)$ is a split injection, 
the composition of the maps $\bigwedge^r_R M \simeq \bigcap^r_R M \to \bigcap^r_R H^1_{\cF}(K, T)$ 
is injective. 
Thus uniqueness follows from $\epsilon_{1} = l_1(\theta(\epsilon) c_{1} \wedge \cdots \wedge c_{r})$. 
 
To show existence, take a core vertex $\fn \in \cN(\cQ)$. 
Put $\fn = \fq_1 \cdots \fq_{\nu(\fn)}$. 
By Remark \ref{sign}, we have $\Phi_{\fn, 1} = (-1)^{\mu(\fn)}\widetilde{\Phi}_{\fn, 1}$ 
where $\widetilde{\Phi}_{\fn, 1}$ is the map defined in Remark \ref{sign} and $\mu(\fn) = \sum_{i=1}^{\mu(\fn)-1}i$. 
Since $M$ is an injective $R$-module and $H^1_{\cF^\fn}(K, T)$ is a free $R$-module of rank 
$r + \nu(\fn)$ by Lemma \ref{freelemma}, 
we can take elements $c_{r+1}, \ldots, c_{r+\nu(\fd)} \in H^{1}_{\cF^{\fn}}(K, T)$ 
such that $H^{1}_{\cF^{\fn}}(K, T) = \bigoplus_{i=1}^{r+\nu(\fn)}Rc_{i}$. 
We also take a basis $w^* = w_{1}^* \wedge \cdots \wedge w_{\nu(\fn)}^*$ of ${\rm det}(W_{\fn})$ 
such that $w_i^* \in H^1_{/f}(K_{\fq_i} T)^*$ and 
$l_\fn(c_{1} \wedge \cdots \wedge c_{r+\nu(\fn)}) \otimes w^* = \epsilon_\fn$. 
We denote by $w_i \in H^1_{/f}(K_{\fq_i}, T)$ the dual basis of $w_i^*$. 
Let $\theta(\epsilon)'$ be the determinant of 
the composition of the maps 
\begin{align*}
\varphi \colon \bigoplus_{i=r+1}^{r+\nu(\fn)} Rc_{i} \to H^1_{\cF^\fn}(K, T) 
\xrightarrow{\oplus_{\fq \mid \fn}{\rm loc}_{\fq}^{/f}} \bigoplus_{\fq \mid \fn}
H^1_{/f}(K_\fq, T) = \bigoplus_{i=1}^{\nu(\fn)}Rw_{i}. 
\end{align*}
By Theorem \ref{pt} and $c_{1}, \ldots, c_{r} \in H^{1}_{\cF}(K, T)$, we have 
$H^{1}_{\cF^*}(K, T^*(1))^* = {\rm coker}\left( \varphi \right)$. 
By the definition of Fitting ideals, 
we have $\theta(\epsilon)' R = {\rm Fitt}_{0, R}\left( H^{1}_{\cF^*}(K, T^*(1))^* \right)$. 
Applying Proposition \ref{commlem2} to the following commutative diagram 
\begin{align*}
\xymatrix{
0 \ar[r] & M \ar[r] \ar[d] & H_{\cF^{\fn}}(K, T) \ar[r] \ar[d]^{\rm id} 
& \bigoplus_{i = r+1}^{r + \nu(\fn)} R c_{i} 
\ar[r] \ar[d]^{\varphi} & 0
\\
0 \ar[r] & H_{\cF}^{1}(K, T) \ar[r] & H_{\cF^{\fn}}^{1}(K, T) \ar[r] & \bigoplus_{i=1}^{\nu(\fn)} R w_{i}, 
}
\end{align*}
we have the following commutative diagram:
\begin{align*}
\xymatrix{
\bigcap^{r + \nu(\fn)}_R H_{\cF^{\fn}}(K, T) \ar[r] \ar[d]^{\rm id} 
& \bigcap^{r}_R M \otimes_R \det\left( \bigoplus_{i = r+1}^{r + \nu(\fn)} Rc_i \right) 
\ar[d]_{j \otimes \det(\varphi)} 
\\
\bigcap^{r + \nu(\fn)}_R H_{\cF^{\fn}}(K, T) \ar[r] 
& \bigcap^{r}_R H_{\cF}^{1}(K, T) \otimes_R \det\left( \bigoplus_{i = 1}^{\nu(\fn)} Rw_i \right)
}
\end{align*}
where $j \colon \bigcap^r_R M \to \bigcap^r_R H^1_{\cF}(K, T)$ is the canonical injection. 
Since the right vertical map sends 
$l_1(c_{1} \wedge \cdots \wedge c_{r}) \otimes (c_{r+1} \wedge \cdots \wedge c_{r+\nu(\fn)})$ to 
$l_1(\theta(\epsilon)' c_{1} \wedge \cdots \wedge c_{r}) \otimes (w_1 \wedge \cdots \wedge w_r)$, 
we have $\epsilon_1 = \Phi_{\fn, 1}(\epsilon_\fn) 
= (-1)^{\mu(\fn)} \widetilde{\Phi}_{\fn, 1}(l_1(c_1 \wedge \cdots \wedge c_{r+\nu(\fn)}) \otimes w^*)
= (-1)^{\mu(\fn)} l_1(\theta(\epsilon)' c_1 \wedge \cdots \wedge c_r)$. 
This completes the proof. 
\end{proof}

Let $\cQ$ be a subset of $\cP$ such that $\cN(\cQ)$ has a core vertex 
and $S$ a zero dimensional Gorenstein local ring. 
Suppose that $\cF$ is cartesian and that there is a surjective ring homomorphism $\pi \colon R \to S$. 
Take a core vertex $\fn \in \cN(\cQ)$ for $\cF$. 
By Lemma \ref{cartesiansurj}, Remark \ref{corerem} (3), and Lemma \ref{freelemma}, we have a map 
\begin{align*}
\bigcap^{r + \nu(\fn)}_R H^1_{\cF^\fn}(K, T) \simeq &\bigwedge^{r + \nu(\fn)}_R H^1_{\cF^\fn}(K, T) 
\to \bigwedge^{r + \nu(\fn)}_R H^1_{\cF^\fn_S}(K, T \otimes_R S) \simeq \bigcap^{r + \nu(\fn)}_R 
H^1_{\cF^\fn_S}(K, T \otimes_R S). 
\end{align*}
If $r = \chi(\cF) \geq 0$, we get a map 
\begin{align*}
\phi_\pi \colon {\bf SS}_{r}(T, \cF, \cQ) &\simeq X_\fn(T, \cF) 
\to X_\fn(T \otimes_R S, \cF_S) \simeq {\bf SS}_{r}(T \otimes_R S, \cF, \cQ) 
\end{align*}
by Theorem \ref{core}. 
It is easy to see that the map $\phi_\pi$ is independent of the choice of the core vertex $\fn \in \cN(\cQ)$. 

\begin{proposition}\label{surj}
Suppose that $\cF$ is cartesian and $r = \chi(\cF) \geq 0$. 
Let $\cQ$ be a subset of $\cP$ such that $\cN(\cQ)$ has a core vertex, $S$ a zero dimensional Gorenstein local ring, 
and $\pi \colon R \to S$ a surjective ring homomorphism. 
\begin{itemize}
\item[(1)] The map $\phi_{\pi}$ induces an isomorphism 
\begin{align*}
{\bf SS}_{r}(T, \cF, \cQ) \otimes_R S \simeq {\bf SS}_{r}(T \otimes_R S, \cF_S, \cQ).
\end{align*}
\item[(2)] We have $I_{i}(\epsilon)S = I_{i}(\phi_{\pi}(\epsilon))$ 
for any $\epsilon \in {\bf SS}_r(T, \cF, \cQ)$ and non-negative integer $i$. 
\end{itemize}
\end{proposition}
\begin{proof}
Let $\fn \in \cN(\cQ)$ be a core vertex. 
Note that $\cF_S$ is cartesian by Lemma \ref{cartesiansurj}. 
Hence by Lemma \ref{basis} and Lemma \ref{freelemma}, 
the map $X_\fn(T, \cF) \to X_{\fn}(S, \cF_S)$ induces an isomorphism 
$X_{\fn}(T, \cF) \otimes_R S \simeq X_{\fn}(S, \cF_S)$. 
Hence the map $\phi_\pi$ induces an isomorphism 
${\bf SS}_{r}(T, \cF, \cQ) \otimes_R S \simeq {\bf SS}_{r}(S, \cF_S, \cQ)$. 

We will prove the assertion (2). 
Let $\fm \in \cN(\cQ)$. 
Since $\cN(\cQ)$ has a core vertex, there is a core vertex $\fn \in \cN(\cQ)$ with $\fm \mid \fn$. 
We fix an isomorphism $H^1_{/f}(K_\fq, T) \simeq R$ for each prime $\fq \mid \fn$. 
Let $\star \in  \{\fm, \fn \}$. 
Using these isomorphisms, we regard the element $\epsilon_\star$ (resp. $\phi_{\pi}(\epsilon)_\star$) 
as an element of $\bigcap^{r+\nu(\star)}_R H^1_{\cF^\star}(K, T)$ 
(resp. $\bigcap^{r+\nu(\star)}_S H^1_{\cF^\star_S}(K, T \otimes_R S)$). 
We put $H_\star = H^1_{\cF^\star}(K, T)^*$ and $H_\star' = H^1_{\cF^\star_S}(K, T \otimes_R S)^*$. 
Applying Lemma \ref{base} to the exact sequence $W_{\fn\fm^{-1}} \to H_\fn \to H_\fm \to 0$, 
we get a map 
\begin{align*}
\psi \colon \bigwedge^{r + \nu(\fm)}_R H_\fm \simeq 
\bigwedge^{r + \nu(\fm)}_R H_\fm \otimes_R \det(W_{\fn\fm^{-1}}) \to \bigwedge^{r+\nu(\fn)}_R H_\fn 
\end{align*}
where the first isomorphism is induced by the fixed isomorphisms $H^1_{/f}(K_\fq, T) \simeq R$. 
In the same way, we also get a map $\psi' \colon \bigwedge^{r + \nu(\fm)}_R H_\fm' 
\to \bigwedge^{r+\nu(\fn)}_R H_\fn'$. 
Then we have the following commutative diagram:
\begin{align*}
\xymatrix{
\bigwedge^{r+\nu(\fm)}_R H_\fn \ar[r]^{j} \ar[d]^-{s} & 
\bigwedge^{r+\nu(\fm)}_R H_\fm \ar[r]^-{\psi}  & 
\bigwedge^{r+\nu(\fn)}_R H_\fn \ar[r]^-{\epsilon_\fn} \ar[d]^-{t} & R \ar[d]^{\pi}
\\
\bigwedge^{r+\nu(\fm)}_S H_\fn' \ar[r]^{j'}  & 
\bigwedge^{r+\nu(\fm)}_S H_\fm' \ar[r]^-{\psi'}  & 
\bigwedge^{r+\nu(\fn)}_S H_\fn' \ar[r]^-{\phi_{\pi}(\epsilon)_\fn}  & S
}
\end{align*}
where the maps $s$ and $t$ are induced by a map 
\begin{align*}
H_\fn \to H_\fn \otimes_R S \simeq (H^1_{\cF^\fn}(K, T) \otimes_R S)^* \simeq H'_\fn
\end{align*} 
and the map $j$ (resp. $j'$) is induced by the injective map 
$H^1_{\cF^\fm}(K, T) \to H^1_{\cF^\fn}(K, T)$ 
(resp. $H^1_{\cF^\fm_S}(K, T \otimes_R S) \to H^1_{\cF^\fn_S}(K, T \otimes_R S)$). 
By the definition of Stark system, we have $\phi_{\pi}(\epsilon)_\fm = \phi_{\pi}(\epsilon)_\fn \circ \psi'$ and 
$\epsilon_\fm = \epsilon_\fn \circ \psi$. 
Since the maps $s$, $j$, and $j'$ are surjective, we have 
\begin{align*}
{\rm im}(\phi_{\pi}(\epsilon)_\fm) 
= {\rm im}\left( \phi_{\pi}(\epsilon)_\fm \circ j' \circ s \right) 
= {\rm im} \left( \pi \circ \epsilon_\fm \circ j \right) 
= {\rm im}(\epsilon_\fm)S. 
\end{align*}
This completes the proof. 
\end{proof}

\subsection{Stark Systems over Principal Artinian Local Rings}

Suppose that $R$ is a principal artinian local ring, $\cF$ is cartesian, and $r = \chi(\cF) \geq 0$.  
We will show that there is a canonical isomorphism 
from the module of Stark systems defined in \cite{MR2} to 
the module of Stark systems defined in the previous subsection. 

Let $\cQ$ be an infinite subset of $\cP_{L_R(R)}(\cF)$ such that $\cN(\cQ)$ has a core vertex. 
We put 
\begin{align*}
W'_\fn := \bigoplus_{\fq \mid \fn} H^1_{\rm tr}(K_\fq, T)^* 
= \bigoplus_{\fq \mid \fn}\Hom_R \left( H^1_{\rm tr}(K_\fq, T), R \right)
\end{align*}
for each ideal $\fn \in \cN(\cQ)$. 
By Lemma \ref{local}, the composition of the maps
\begin{align*}
H^1_{\rm tr}(K_\fq, T) \to H^1(K, T) \xrightarrow{{\rm loc}^{/f}_\fq} H^1_{/f}(K_\fq, T)
\end{align*}
is an isomorphism for any prime $\fq \in \cQ$. 
Hence we have a canonical isomorphism $j_{\fn} \colon W'_{\fn} \to W_{\fn}$ for any ideal $\fn \in \cN(\cQ)$. 

Let $Y_{\fn} := \bigwedge^{r + \nu(\fn)}_R H^{1}_{\cF^{\fn}}(K, T) \otimes_R \det(W'_{\fn})$ and 
$\Psi_{\fn, \fm} \colon Y_\fn \to Y_\fm$ the map defined in \cite[Definition 6.3]{MR2}. 
Then the module ${\bf SS}_r (T, \cF, \cQ)' $ of Stark systems defined in \cite{MR2} 
is the inverse limit 
\begin{align*}
{\bf SS}_r (T, \cF, \cQ)' := \varprojlim_{\fn \in \cN(\cQ)} Y_{\fn} 
\end{align*}
with respect to the maps $\Psi_{\fn, \fm}$. 

\begin{lemma}\label{comparison}
Let $0 \to N \to M \xrightarrow{h} F$ be an exact sequence of finitely generated $R$-modules, 
$F$ a free $R$-module of rank one, and $s$ a positive integer. 
Then the following diagram commutes
\begin{align*}
\xymatrix@C=60pt{
\bigwedge^s_R M \ar[r]^-{\widehat{h}} \ar[d]^-{l_M} & \bigwedge^{s-1}_R N \otimes_R F \ar[d]^-{l_N \otimes {\rm id}_F}
\\
\bigcap^s_R M \ar[r]^-{(-1)^{s-1}\widetilde{h}} & \bigcap^{s-1}_R N \otimes_R F
}
\end{align*}
where $l_N$ and $l_M$ are the canonical maps, $\widehat{h}$ is the map defined in \cite[Proposition A.1]{MR2}, 
and $\widetilde{h}$ is the dual of the map defined by the exact sequence 
$F^* \xrightarrow{h^*} M^* \to N^* \to 0$ using Lemma \ref{base}. 
\end{lemma}
\begin{proof}
We may assume that $F = R$. 
Since $R$ is a principal artinian local ring, we can write $M = Rm \oplus N_0$ and $N = Im \oplus N_0$ 
for some ideal $I$ of $R$. Then $\bigwedge^s_R M$ is generated by the set 
\begin{align*}
E = \{m_1 \wedge \cdots \wedge m_s \mid m_2, \ldots, m_{s} \in N_0, m_1 \in \{m\} \cup N_0\}. 
\end{align*}
Let $f = f_1 \wedge \cdots \wedge f_{s-1} \in \bigwedge^{s-1}_R N^*$ and 
$\widetilde{f}_i$ a lift of $f_i$ in $M^*$. 
Note that 
$\widetilde{h} \circ l_M(m_1 \wedge \cdots \wedge m_s)(f_1 \wedge \cdots \wedge f_{s-1}) = 
\det\left( (\widetilde{f}_i(m_j))_{1 \leq i \leq s, 1 \leq j \leq s} \right)$ where $\widetilde{f}_s = h$. 
Thus if $m = m_1 \wedge \cdots \wedge m_s \in E$, then 
\begin{align*}
(-1)^{s-1} \widetilde{h} \circ l_M(m)(f) 
&= (-1)^{s-1} \det\left( \left(\widetilde{f}_i(m_j) \right)_{1 \leq i \leq s, 1 \leq j \leq s} \right) 
\\
&= h(m_1) \det\left( \left(\widetilde{f}_i(m_j) \right)_{1 \leq i < s, 2 \leq j \leq s} \right)
\\
& = h(m_1) l_N(m_2 \wedge \cdots \wedge m_s)(f)
\\
& = l_N \circ \widehat{h}(m)(f) 
\end{align*}
where the second equality follows from $\widetilde{f}_s(m_i) = h(m_i) = 0$ for $2 \leq i \leq s$, 
the third equality follows from $m_2, \ldots, m_s \in N_0 \subseteq N$, 
and the last equality follows from $\widehat{h}(m) = m_2 \wedge \cdots \wedge m_s \otimes h(m_1)$. 
This completes the proof. 
\end{proof}

Let $\fn \in \cN(\cQ)$ and 
$l_{\fn} \colon \bigwedge^{r+\nu(\fn)}_R H^{1}_{\cF^{\fn}}(K, T) \to \bigcap^{r+\nu(\fn)}_R H^{1}_{\cF^{\fn}}(K, T)$ 
the canonical map. 
Put $\mu(\fn) = \sum_{i=1}^{r+\nu(\fn)-1}i$. 
Define a map $C_{\fn} \colon Y_{\fn} \to X_{\fn}$ by $C_\fn = (-1)^{\mu(\fn)} l_\fn \otimes j_{\fn}$. 

\begin{proposition}\label{cstark}
The maps $C_{\fn} \colon Y_{\fn} \to X_{\fn}$ induce an isomorphism 
\begin{align*}
C \colon {\bf SS}_{r}(T, \cF, \cQ)' \simeq {\bf SS}_{r}(T, \cF, \cQ). 
\end{align*}
\end{proposition}
\begin{proof}
By Lemma \ref{comparison} and the definition of 
the transition maps $\Psi_{\fn, \fm}$ and $\Phi_{\fn, \fm}$, the maps $C_\fn$ induces a map 
$C \colon {\bf SS}_{r}(T, \cF, \cQ)' \to {\bf SS}_{r}(T, \cF, \cQ)$. 
By Lemma \ref{freelemma} or \cite[Proposition 3.3]{MR2}, the map $C_\fn \colon Y_{\fn} \to X_{\fn}$ is an isomorphism 
for any core vertex $\fn \in \cN(\cQ)$. 
Thus by Theorem \ref{core} and \cite[Theorem 6.7]{MR2}, the map $C$ is an isomorphism. 
\end{proof}

\begin{lemma}\label{compalem}
Let $A$ be a zero dimensional Gorenstein local ring, $M$ a finitely generated $A$-module, and $s$ a positive integer. 
Let 
\begin{align*}
l_M \colon \bigwedge^s_A M \to \bigcap^s_A M = \Hom_R \left( \bigwedge_A^s M^*, A \right) 
\end{align*}
denote the canonical map and $x \in \bigwedge^s_A M$. 
\begin{itemize}
\item[(1)] Let $J$ be an ideal of $A$. If $x \in J \bigwedge^s_A M$, 
then ${\rm im}(l_M(x)) \subseteq J$. 
\item[(2)] Suppose that there is a free $A$-submodule $M' \subseteq M$ of rank $s$ such that 
$x \in {\rm im}(\bigwedge^{s}_A M' \to \bigwedge^s_A M)$. Then $x \in {\rm im}(l_M(x))\bigwedge^s_A M$. 
\end{itemize}
\end{lemma}
\begin{proof}
Since ${\rm im}(l_M(ay)) = {\rm im}(a \cdot l_M(y)) = a \cdot {\rm im}(l_M(y))$ for any 
$y \in \bigwedge^s_A M$ and $a \in A$, the assertion (1) holds. 
We will show the assertion (2). 
Since $A$ is a zero dimensional Gorenstein local ring, we can write $M = M' \oplus N$ 
for some $A$-submodule $N \subseteq M$. 
Let $y_1, \ldots, y_s$ be a basis of $M'$. We take the element $y_i^* \in M^*$ such that $y_i^*(y_i) = 1$, 
$y_i^*(y_j) = 0$ if $i \neq j$, and $y_i^*(N) = 0$.  
We write $x = ay_1 \wedge \cdots \wedge y_s$ for some $a \in A$. 
Then we have $l_M(x)(y_1^* \wedge \cdots \wedge y_s^*) = a$. 
This completes the proof.  
\end{proof}

\begin{lemma}\label{takebasis}
Let $0 \to N \to M \xrightarrow{h} F$ be an exact sequence of finitely generated $R$-modules, 
$F$ a free $R$-module of rank one, and $s$ a positive integer. 
Let $\widehat{h} \colon \bigwedge^s_R M \to \bigwedge^{s-1}N \otimes_R F$ denote the map 
defined in \cite[Proposition A.1]{MR2} and $x \in \bigwedge^s_R M$. 
Suppose that there is a free $R$-submodule $M' \subseteq M$ of rank $s$ such that 
$x \in {\rm im}\left( \bigwedge^{s}_R M' \to \bigwedge^s_R M \right)$. 
Then there is a free $R$-submodule $N' \subseteq M' \cap N$ of rank $s-1$ such that 
$\widehat{h}(y) \in {\rm im}\left( \bigwedge^{s-1}_R N' \to \bigwedge^{s-1}_R N \right) \otimes_R F 
\subseteq \bigwedge^{s-1}_R N \otimes_R F$. 
\end{lemma}
\begin{proof}
Since $R$ is a principal artinian local ring, there is a basis $y_1, \ldots, y_s$ of $M'$ 
such that $y_1, \ldots, y_{s-1} \in N$. 
Put $N' = \sum_{i=1}^{s-1}Ry_i$. 
Then $N'$ is a free $R$-submodule of $M' \cap N$ of rank $s-1$ and 
$\widehat{h}(y_1 \wedge \cdots \wedge y_s) = 
(-1)^s y_1 \wedge \cdots \wedge y_{s-1} \otimes h(y_s) \in {\rm im}\left( \bigwedge^{s-1}_R N' 
\to \bigwedge^{s-1}_R N \right) \otimes_R F$. 
\end{proof}

The following proposition says that Theorem \ref{control} 
is a generalization of \cite[Theorem 8.5]{MR2}. 

\begin{proposition}\label{gen}
Let $\epsilon'$ be a basis of ${\bf SS}_{r}(T, \cF, \cQ)'$ and $k = L_R(R)$. 
Then 
\begin{align*}
\partial \varphi_{\epsilon'}(t) 
= \max \{ n \in \{0,1, \ldots, k-1, \infty\} \mid I_{t}(C(\epsilon')) \subseteq \fm_R^n\}. 
\end{align*}
Here $\partial \varphi_{\epsilon'} \colon \bZ_{\geq 0} \to \bZ_{\geq 0} \cup \{ \infty \}$ 
is the map defined in \cite[Definition 8.1]{MR2}. 
\end{proposition}
\begin{proof}
We fix an isomorphism $H^1_{/f}(K_\fq, T) \simeq R$ for each prime $\fq \in \cQ$. 
Then we get an isomorphism $i_\fn \colon X_\fn \simeq \bigcap^{r + \nu(\fn)}_R H^1_{\cF^\fn}(K, T)$ 
for each ideal $\fn \in \cN(\cQ)$. 
Let $c$ be a non-negative integer and $\fm \in \cN(\cQ)$. 
Then we only need to show that the following assertions are equivalent. 
\begin{itemize}
\item[(i)] $\epsilon_\fm' \in \fm_R^c Y_\fm$. 
\item[(ii)] ${\rm im} \left( i_\fm(C(\epsilon')_\fm) \right) = {\rm im} \left( i_\fm(C_\fm(\epsilon'_\fm)) \right) 
\subseteq \fm^c_R$. 
\end{itemize}
By Lemma \ref{compalem} (1), (i) implies (ii). 
We will show that (ii) implies (i). 
Since we assume that $\cN(\cQ)$ has a core vertex, there is a core vertex $\fn \in \cN(\cQ)$ with $\fm \mid \fn$. 
Since $H^1_{\cF^\fn}(K, T)$ is a free $R$-module of rank $r + \nu(\fn)$, by Lemma \ref{takebasis}, 
there is a free $R$-submodule $M \subseteq H^1_{\cF^\fm}(K, T)$ of rank $r + \nu(\fm)$ such that 
$\epsilon_\fn' \in {\rm im}\left(\bigwedge^{r+\nu(\fm)}_R M \to \bigwedge^{r+\nu(\fm)}_R H^1_{\cF^\fm}(K, T) \right) 
\otimes_R \det(W_\fm')$. Thus by Lemma \ref{compalem} (2), we see that (ii) implies (i). 
\end{proof}

\section{Stark Systems over Gorenstein Local Rings}

In this section, we assume that $R$ is a complete Gorenstein local ring and 
that $T$ satisfies Hypothesis \ref{standardhyp}. 
Since $R$ is a complete Gorenstein local ring, there is an increasing sequence 
$J_1 \subseteq J_2 \subseteq J_3 \subseteq \cdots$ of ideals of $R$ 
such that $R/J_n$ is a zero dimensional Gorenstein local ring and 
a canonical map $R \to \varprojlim_{n}R/J_n$ is an isomorphism 
where the inverse limit is taken with respect to canonical maps $R/J_{n+1} \to R/J_n$. 
In fact, since a Gorenstein local ring is Cohen-Macaulay, there is a 
regular sequence $x_1, \ldots, x_d \in \fm_R$ such that $\sqrt{(x_1, \ldots, x_d)} = \fm_R$. 
Put $J_n = (x_1^n, \ldots, x_d^n)$ for each positive integer $n$. 
Since $x_1^n, \ldots, x_d^n$ is a regular sequence and $\sqrt{J_n} = \fm_R$, 
$R/J_n$ is a zero dimensional Gorenstein local ring. 
Furthermore the map $R \to \varprojlim_{n}R/J_n$ is an isomorphism since $R$ is a complete noetherian local ring. 
Throughout this section, we fix such a sequence $\{J_n\}_{n \geq 1}$ of ideals of $R$. 
 
Let $n$ be a positive integer and $\cF_{n}$ a Selmer structure on $T/J_n T$ such that 
\begin{align*}
H^1_{\cF_{n}}(K_\fq, T/J_{n}T) = H^1_{(\cF_{n+1})_{R/J_{n}}}(K_\fq, T/J_{n}T) 
\end{align*}
for any prime $\fq$ of $K$. 
Then we have $r := \chi(\cF_n) = \chi(\cF_{n+1})$ for any positive integer $n$. 
In order to use the results of the previous section, we assume that $r \geq 0$ and that 
the Selmer structure $\cF_{n}$ is cartesian for all positive integer $n$. 
Put $\cF = \{\cF_n \}_{n \geq 1}$. 

\begin{remark}\label{indeprem}
Let $S$ be a zero dimensional Gorenstein local ring and $\pi \colon R \to S$ a surjective ring homomorphism. 
Since the kernel of $\pi$ is an open ideal, there is a positive integer $n$ such that $I_n \subseteq \ker(\pi)$. 
Hence the collection of the Selmer structures $\cF$ defines a Selmer structure $\cF_S$ on $T \otimes_R S$. 
Note that $\chi(\cF_S) = r$ and $\cF_S$ is cartesian by Lemma \ref{gorencar}.  
\end{remark}

\begin{remark}\label{indeprem2}
Let $S$ be a complete Gorenstein local ring and $\pi \colon R \to S$ a surjective ring homomorphism. 
Take an increasing sequence $\{J_{S, n} \}_{n \geq 1}$ of ideals of $S$ 
such that $S/J_{S, n}$ is a zero dimensional Gorenstein local ring and the map 
$S \to \varprojlim_{n} S/J_{S, n}$ is an isomorphism. 
By Remark \ref{indeprem}, we get a Selmer structure $\cF_{S, n}$ on $T \otimes_R S/J_{S, n}$ 
for each positive integer $n$. 
Then we have $\chi(\cF_{S, n}) = r$, the Selmer structure $\cF_{S, n}$ is cartesian, and 
\begin{align*}
H^1_{\cF_{S, n}}(K_\fq, T \otimes_R S/J_{S, n}S) = H^1_{(\cF_{S, n+1})_{R/J_{n}}}(K_\fq, T \otimes_R S/J_{S, n}S) 
\end{align*}
for any prime $\fq$ of $K$ and positive integer $n$. 
We denote by $\cF_S$ the collection $\{\cF_{S, n}\}_{n \geq 1}$ of the Selmer structures. 
\end{remark}

\begin{example}\label{iwasawa}
Let $\cO$ be the ring of integers of a finite extension of the field $\bQ_p$ 
and $\cT$ a free $\cO$-module of finite rank with an $\cO$-linear 
continuous $G_K$-action which is unramified outside a finite set of places of $K$. 
Let $d$ be a positive integer and $L$ a free $R := \cO[[X_1, \ldots, X_{d}]]$-module of rank one with 
an $R$-linear continuous $G_K$-action which is unramified outside primes above $p$. 
Put $T = \cT \otimes_\cO L$ and $J_n = (p^n, X_1^n, \ldots, X_d^n)$. 
Suppose that the following conditions hold: 
\begin{itemize}
\item the module $\cT$ satisfies the conditions (H.1), (H.2), and (H.3),
\item the module $L \otimes_R R/(X_1, \cdots, X_d)$ has trivial $G_K$-action, 
\item the module $(\cT \otimes_{\bZ_p} \bQ_p/\bZ_p)^{\cI_\fq}$ is divisible for every prime $\fq \nmid p$ of $K$,  
\item $H^2(K_\fp, \cT) = 0$ for each prime $\fp \mid p$ of $K$.
\end{itemize}
For a positive integer $n$, we define a Selmer structure $\cF_n$ on $T/J_n T$ by the following data:
\begin{itemize}
\item $\Sigma(\cF_{n}) := \{v \mid p \infty \} \cup \{ \fq \mid \text{$\cT$ is ramified at $\fq$} \}$, 
\item $H^1_{\cF_n}(K_\fq, T/ J_n T) = H^1_{\rm ur}(K_\fq, T/J_n T)$ 
for each prime $\fq \nmid p$ of $K$, 
\item $H^1_{\cF_n}(K_\fp, T/ J_n T) = H^1(K_\fp, T/ J_n T)$ 
for each prime $\fp \mid p$ of $K$. 
\end{itemize}
Since $(\cT \otimes_{\bZ_p} \bQ_p/\bZ_p)^{\cI_\fq}$ is divisible and 
$L^{\cI_\fq} = L$ for every prime $\fq \nmid p$ of $K$, a canonical map 
$(T/J_{n+1} T)^{\cI_\fq} \to (T/J_n T)^{\cI_\fq}$ is surjective. 
Note that the cohomological dimension of $\cD_\fq/\cI_\fq \simeq \widehat{\bZ}$ is one. 
Therefore a canonical map $H^1_{\rm ur}(K_\fq, T/J_{n+1} T) \to H^1_{\rm ur}(K_\fq, T/J_n T)$ is surjective. 
By the inflation-restriction exact sequence, we have an isomorphism 
$H^1_{/\cF_n}(K_\fq, T/J_n T) \simeq H^1(\cI_\fq, T/J_n T)^{{\rm Fr_\fq} = 1}$. 
Hence we see that the map 
\begin{align*}
H^1_{/\cF_n}(K_\fq, T/\fm_R T) \to H^1_{/\cF_n}(K_\fq, T/J_n T)
\end{align*}
induced by an injection $\Bbbk \to R$ is injective. 
Let $\fp \mid p$ be a prime of $K$. 
Since the cohomological dimension of $\cD_\fp$ is $2$, we have 
\begin{align*}
H^2(K_\fp, T) \otimes_R R/(X_1, \ldots, X_d) \simeq H^2(K_\fp, \cT). 
\end{align*}
Since $H^2(K_\fp, \cT) = 0$, we have $H^2(K_\fp, T) = 0$ by Nakayama's Lemma, 
and a map $H^1(K_\fp, T/ J_{n+1} T) \to H^1(K_\fp, T/ J_n T)$ is surjective for any positive integer $n$. 
Furthermore, by \cite[Theorem 5.4]{MR2}, we have 
\begin{align*}
\chi(\cF_{n}) = \sum_{v \mid \infty} {\rm corank}_\cO \left( H^0(K_v, \cT^*(1)) \right) \geq 0
\end{align*}
where $v$ runs through all the infinite places of $K$. 
Hence the collection of the Selmer structures $\{ \cF_n \}_{n \geq 1}$ satisfies all the conditions 
in the paragraph before Remark \ref{indeprem}. 
\end{example}

\subsection{Stark Systems over Gorenstein Local Rings}
We fix a decreasing sequence $\cQ_1 \supseteq \cQ_2 \supseteq \cQ_3 \supseteq \cdots$ such that 
$\cQ_n$ is an infinite subset of $\cP_{L_R(R/J_n)}(\cF_n)$ and $\cN(\cQ)$ has a core vertex for $\cF_n$. 
For example, we take 
$\cQ_n = \cP_{L_R(R/J_n)}(\cF_n) \setminus \left( \Sigma(\cF_1) \cup \cdots \cup \Sigma(\cF_{n}) \right)$ 
for each positive integer $n$. 
We denote by $\cQ$ the collection of the sequence $\{ \cQ_n \}_{n \geq 1}$. 

Let $n$ be a positive integer. 
Since $\cN(\cQ_{n+1})$ has a core vertex for $\cF_{n+1}$, the map $R/J_{n+1} \to R/J_{n}$ induces a map 
${\bf SS}_{r}(T/J_{n +1} T, \cF_{n+1}, \cQ_{n+1}) \to {\bf SS}_{r}(T/J_{n}T, \cF_n, \cQ_{n+1})$ 
as in the paragraph before Proposition \ref{surj}. 
Since a restriction map ${\bf SS}_{r}(T/J_n T, \cF_{n}, \cQ_{n}) \to {\bf SS}_{r}(T/J_{n}T, \cF_n, \cQ_{n+1})$ 
is an isomorphism by Theorem \ref{core}, we get a map 
\begin{align*}
\phi_{n+1, n} \colon {\bf SS}_{r}(T/J_{n +1} T, \cF_{n+1}, \cQ_{n+1}) \to {\bf SS}_{r}(T/J_{n}T, \cF_n, \cQ_{n}). 
\end{align*}
We define the module ${\bf SS}_{r}(T, \cF, \cQ)$ of Stark systems of rank $r$ for $\cF$ by 
\begin{align*}
{\bf SS}_{r}(T, \cF, \cQ) &:= \varprojlim_{n \geq 1} {\bf SS}_r(T/J_n T, \cF_n, \cQ_n) 
\end{align*}
where the inverse limit is taken with respect to the maps $\phi_{n+1, n}$. 

Let $S$ be a complete Gorenstein local ring and $\pi \colon R \to S$ a surjective ring homomorphism. 
Take an increasing sequence $J_{S, 1} \subseteq J_{S, 2} \subseteq J_{S, 3} \subseteq \cdots$ 
of ideals of $S$ such that $S/J_{S, n}$ is a zero dimensional Gorenstein local ring 
and the map $S \to \varprojlim_{n} S/J_{S, n}$ is an isomorphism. 
Let $\cF_S$ be the collection of the Selmer structures defined in Remark \ref{indeprem2}. 
For a positive integer $n$, we fix a positive integer $n_S$ such that $n_S \geq n$ and that 
the map $\pi \colon R \to S$ induces a map $ \pi_n \colon R/J_{n_S} \to S/J_{S, n}$. 
Then we have the map 
\begin{align*}
\phi_{\pi_n} \colon {\bf SS}_{r}(T/J_{n_S} T, \cF_{n_S}, \cQ_{n_S}) \to 
{\bf SS}_{r}(T \otimes_R S/J_{S, n}, \cF_{S, n}, \cQ_{n_S}). 
\end{align*}
Hence we get a canonical map 
\begin{align*}
\phi_\pi \colon {\bf SS}_{r}(T, \cF, \cQ) \to 
{\bf SS}_{r}(T \otimes_R S, \cF_{S}, \{ \cQ_{n_S} \}_{n \geq 1}). 
\end{align*}
For simplicity, we also denote by $\cQ$ the subsequence $\{\cQ_{n_S}\}_{n \geq 1}$ of $\cQ = \{\cQ_n \}_{n \geq 1}$. 

\begin{theorem}\label{main1}
\begin{itemize}
\item[(1)] The $R$-module ${\bf SS}_{r}(T, \cF, \cQ)$ is free of rank one. 
\item[(2)] Let $S$ be a complete Gorenstein local ring and $\pi \colon R \to S$ a surjective ring homomorphism. 
Then the map $\phi_\pi$ induces an isomorphism 
\begin{align*}
{\bf SS}_{r}(T, \cF, \cQ) \otimes_R S \simeq {\bf SS}_{r}(T \otimes_R S, \cF_S, \cQ). 
\end{align*}
\end{itemize}
\end{theorem}
\begin{proof}
This theorem follows immediately from Theorem \ref{core} and Proposition \ref{surj} (1). 
\end{proof}

Let $i$ be a non-negative integer and $\epsilon \in {\bf SS}_{r}(T, \cF, \cQ)$. 
For a positive integer $n$, let $\phi_n \colon {\bf SS}_r(T, \cF, \cQ) \to {\bf SS}_r(T/J_nT, \cF_n, \cQ_n)$ 
denote the projection map. 
By Proposition \ref{surj} (2), we can define an ideal $I_{i}(\epsilon)$ of $R$ by 
\begin{align*}
I_{i}(\epsilon) := \varprojlim_{n \geq 1}I_{i}(\phi_{n}(\epsilon)) 
= \varprojlim_{n \geq 1} \left( \sum_{\fn \in \cN(\cQ_n), \nu(\fn) = i} {\rm im}(\phi_n(\epsilon)_\fn) \right) 
\end{align*}
where the inverse limit is taken with respect to the maps $R/J_{n+1} \to R/J_n$. 

\begin{definition}
We define the dual Selmer group $H^{1}_{\cF^*}(K, T^*(1))$ associated with the collection of the Selmer structures $\cF$ by 
\begin{align*}
H^{1}_{\cF^*}(K, T^*(1)) := \varinjlim_{n \geq 1} H^{1}_{\cF_n^*}(K, (T/J_n T)^*(1))
\end{align*}
where the injective limit is taken with respect to the maps induced 
by the Cartier duals of the maps $T/J_{n+1} T \to T/J_{n} T$. 
\end{definition}

Set $M^\vee := \Hom(M, \bQ_p/\bZ_p)$ for any $\bZ_p$-module $M$. 
The following theorem is our main result. 

\begin{theorem}\label{main2}
Let $i$ be a non-negative integer. 
\begin{itemize}
\item[(1)] If $\epsilon$ is a basis of ${\bf SS}_{r}(T, \cF, \cQ)$, then we have 
\begin{align*}
I_{i}(\epsilon) = {\rm Fitt}_{i, R}\left( H^{1}_{\cF^*}(K, T^*(1))^\vee \right). 
\end{align*}
\item[(2)] Let $S$ be a complete Gorenstein local ring and $\pi \colon R \to S$ a surjective ring homomorphism. 
Then we have $I_i(\epsilon)S = I_i(\phi_\pi(\epsilon))$ for any $\epsilon \in {\bf SS}_r(T, \cF, \cQ)$. 
\end{itemize}
\end{theorem}
\begin{proof}
By Lemma \ref{dmult}, the canonical map 
\begin{align*}
H^{1}_{\cF^*}(K, T^*(1))^\vee \otimes_R R/J_n \to 
H^{1}_{\cF_n^*}(K, (T/J_n T)^*(1))^\vee
\end{align*}
is an isomorphism for any positive integer $n$. 
Thus we have 
\begin{align*}
{\rm Fitt}_{i, R} \left( H^{1}_{\cF^*}(K, T^*(1))^\vee \right) R/J_n 
&= {\rm Fitt}_{i, R/J_n} \left(H^{1}_{\cF_n^*}(K, (T/J_n T)^*(1))^\vee \right) 
\\
&= {\rm Fitt}_{i, R/J_n} \left(H^{1}_{\cF_n^*}(K, (T/J_n T)^*(1))^* \right) 
\\
&= I_{i}(\phi_{n}(\epsilon))
\\
&= I_{i}(\epsilon)R/J_n
\end{align*}
where the second equality follows from Remark \ref{dual}, 
the third equality follows from Theorem \ref{control}, and 
the fourth equality follows from Proposition \ref{surj} (2). 
Since $R$ is a complete noetherian local ring, all the ideals of $R$ are closed. 
Hence we have $I_{i}(\epsilon) = {\rm Fitt}_{i, R}(H^{1}_{\cF^*}(K, T^*(1))^\vee)$. 
The second assertion follows from Proposition \ref{surj} (2). 
\end{proof}

\begin{remark}\label{compare-dvr}
Suppose that $R$ is a discrete valuation ring. 
Let ${\bf SS}_r(T, \cF, \cQ)'$ be the module of Stark systems 
defined in \cite[Definition 7.1]{MR2} and $n$ a positive integer. 
Let
\begin{align*}
C_n \colon {\bf SS}_r(T/\fm^n_R T, \cF_n, \cQ_n)' \to {\bf SS}_r(T/\fm^n_R T, \cF_n, \cQ_n)   
\end{align*}
be the isomorphism defined in Proposition \ref{cstark}. 
By the definition of the map $C_n$, we have the following commutative diagram: 
\begin{align*}
\xymatrix{
{\bf SS}_r(T/\fm^{n +1}_R T, \cF_{n+1}, \cQ_{n+1})' \ar[r]^{C_{n+1}} \ar[d] & 
{\bf SS}_r(T/ \fm^{n+1}_R T, \cF_n, \cQ_{n+1}) \ar[d] 
\\
{\bf SS}_r(T/\fm^{n}_R T, \cF_n, \cQ_n)' \ar[r]^{C_{n}} & 
{\bf SS}_r(T/ \fm^{n}_R T, \cF_n, \cQ_n). 
}
\end{align*}
Thus by \cite[Proposition 7.3]{MR2}, the maps $C_n$ induce an isomorphism 
\begin{align*}
C \colon {\bf SS}_r(T, \cF, \cQ)' \simeq {\bf SS}_r(T, \cF, \cQ). 
\end{align*}
Furthermore, if $\epsilon' \in {\bf SS}_r(T, \cF, \cQ)'$ is non-zero, we see that 
\begin{align*}
\partial \varphi_{\epsilon'}(t) = \max \{ n \in \bZ_{\geq 0} \cup \{\infty\} \mid I_t(C(\epsilon')) \subseteq \fm^n_R \}
\end{align*}
for any non-negative integer $t$ by Proposition \ref{gen}. 
This shows that Theorem \ref{main2} implies \cite[Theorem 8.7]{MR2}. 
\end{remark}

\section{Controlling Selmer Groups using $\Lambda$-adic Stark Systems}

Let $\cO$ be the ring of integers of a finite extension of the field $\bQ_p$ 
and $\cT$ a free $\cO$-module of finite rank with 
an $\cO$-linear continuous $G_K$-action which is unramified outside a finite set of places of $K$. 
We write $K_{\infty}$ for the cyclotomic $\bZ_{p}$-extension of $K$. 
Fix a topological generator $\gamma$ of $\Gal(K_{\infty}/K)$. 
Let $R$ be the ring of formal power series $\cO[[X]]$. 
By using the element $\gamma$, we get an isomorphism 
$\cO[[\Gal(K_{\infty}/K)]] \simeq R$, $\gamma \mapsto 1 + X$. 
It induces an $R$-module structure on $\cO[[\Gal(K_\infty/K)]]$. 

Following the notation in \cite[Section 5.3]{MR1}, 
we write $\Lambda$ instead of $R$ and put $\bT := \cT \otimes_\Lambda \cO[[\Gal(K_\infty/K)]]$. 
We assume the following conditions:
\begin{itemize}
\item $(\cT/\varpi \cT)^{G_K} = (\cT^*(1)[\varpi])^{G_K} = 0$ 
and $\cT/\varpi \cT$ is an absolutely irreducible $\Bbbk[G_K]$-module where $\varpi$ is a uniformizer of $\cO$, 
\item the module $\cT$ satisfies the conditions (H.2) and (H.3),
\item the module $(\cT \otimes \bQ_p/\bZ_p)^{\cI_\fq}$ is divisible for every prime $\fq \nmid p$ of $K$,  
\item $H^2(K_\fp, \cT) = 0$ for each prime $\fp \mid p$ of $K$.
\end{itemize} 
Let $n$ be a positive integer and $J_n = (p^n, X^n)$. 
Let $\cF_{n}$ be the Selmer structure on $\bT/J_n \bT$ defined in Example \ref{iwasawa}.  
Note that $\chi(\cF_n) = \sum_{v \mid \infty} {\rm corank}_\cO \left(H^0(K_v, \cT^*) \right)$ 
where $v$ runs through all the infinite places of $K$. 
Set $\cF = \{ \cF_n\}_{n \geq 1}$ and $r = \chi(\cF_n)$. 
We can define the module of Stark systems by 
\begin{align*}
{\bf SS}_r(\bT) := {\bf SS}_r(\bT, \cF, \{ \cP_{n^2}(\cF_n) \}_{n \geq 1}) 
= \varprojlim_{n \geq 1} {\bf SS}_r (\bT/J_n \bT, \cF_{n}, \cP_{n^2}(\cF_n)). 
\end{align*} 
Note that ${\bf SS}_r(\bT)$ is a free $\Lambda$-module of rank one by Theorem \ref{main1}. 
We call an element of ${\bf SS}_r(\bT)$ a $\Lambda$-adic Stark system. 
We say that a $\Lambda$-adic Stark system is primitive 
if it is a basis of ${\bf SS}_r(\bT)$. 

\begin{remark}\label{kolyvagin}
Assume that $r = 1$. 
Let ${\bf KS}(\bT)$ be the module of $\Lambda$-adic Kolyvagin systems defined in \cite[Chapter 5]{MR1} or \cite{Bu}. 
Then we can construct an isomorphism ${\bf SS}_1(\bT) \simeq {\bf KS}(\bT)$ (See \cite[Section 3.1.2]{Bu}). 
Thus we have a map from the module of Euler systems to the module of $\Lambda$-adic Stark systems ${\bf SS}_{1}(\bT)$ 
by \cite[Theorem 5.3.3]{MR1}. 
\end{remark}

\begin{example}\label{cyclo}
Suppose that $K$ is a totally real number field. 
Let $\chi \colon G_K \to \cO^{\times}$ be an even character of finite prime-to-$p$ order 
and let $\ff_{\chi}$ be the conductor of $\chi$. 
Put $\cT = \cO(1) \otimes_\cO \chi^{-1}$.
Assume the following hypotheses:
\begin{itemize}
\item $(p, \ff_\chi)=1$, 
\item $K$ is unramified at all primes above $p$, 
\item $\chi({\rm Frob}_{\fq}) \neq 1$ for any prime $\fp$ of $K$ above $p$. 
\end{itemize}
Then $\cT$ satisfies all the conditions in this section and we have $r = [K \colon \bQ]$ by \cite[Corollary 5.6]{MR2}. 
Thus if $K = \bQ$, then we get a $\Lambda$-adic Stark system from the Euler systems of cyclotomic units. 
By using the analytic class number formula, we show that this Stark system is primitive. 
More generally, K. B\"uk\"uboduk constructed an $\bL$-restricted Kolyvagin system 
from conjectural Rubin-Stark units in \cite{bus}. 
Using this fact, we get a primitive $\Lambda$-adic Stark system which is related to Rubin-Stark units if 
we assume Rubin-Stark conjecture. 
\end{example}

\begin{example}
Suppose that $\cO = \bZ_p$. 
Let $A$ be an abelian variety of dimension $d$ defined over $K$ 
and let $\cT$ be the Tate module $T_p(A) = \varprojlim A[p^n]$. 
Assume the following hypotheses:
\begin{itemize}
\item $A$ has good reduction at all the primes of $K$ above $p$, 
\item the image of $G_K$ in ${\rm Aut}(A[p]) \simeq {\rm GL}_{2d}(\bF_p)$ contains ${\rm GSp}_{2d}(\bF_p)$, 
\item $p \nmid 6 \prod_{\fq \nmid p}c_{\fq}$ where $c_\fq \in \bZ_{>0}$ is the Tamagawa factor at $\fq \nmid p$, 
\item $A^d(K_\fq)[p] = 0$ for any prime $\fp \mid p$ where $A^d$ is the dual abelian variety of $A$. 
\end{itemize}
Then $\cT$ satisfies all the conditions in this section and 
we have $r = d [K \colon \bQ]$ by \cite[Proposition 5.7]{MR2}. 
If $K = \bQ$ and $d = 1$, then the Kato's Euler systems gives rise to a $\Lambda$-adic Stark system 
(See \cite[Theorem 6.2.4]{MR1}, \cite[Proposition 6.2.6]{MR1}, and \cite[Proposition 4.2]{Bu}). 
\end{example}

By using \cite[Proposition B.3.4]{Ru}, we have $H^{1}(K_\fq, \bT) = H^1_{\rm ur}(K_\fq, \bT)$ 
for each prime $\fq \nmid p$ of $K$. 
Hence we can define a Selmer structure $\cF_{\Lambda}$ on $\bT$ by the following data:
\begin{itemize}
\item $\Sigma(\cF_{\Lambda}) = \Sigma := \{v \mid p \infty \} \cup \{\fq \mid \text{$\fq$ is ramified at $\cT$}\}$, 
\item $H^{1}_{\cF_{\Lambda}}(K_\fq, \bT) = H^{1}(K_\fq, \bT)$ for each prime $\fq$ of $K$. 
\end{itemize}

\begin{remark}
Since $H^{1}(K_\fq, \bT) = H^1_{\rm ur}(K_\fq, \bT)$ 
for each prime $\fq \nmid p$ of $K$, we have 
\begin{align*}
H^1_{\cF_{\Lambda}}(K, \bT) = H^1(K, \bT) = H^1(K_{\Sigma}/K, \bT) 
\end{align*}
where $K_\Sigma$ is the maximal extension of $K$ unramified outside $\Sigma$ and 
$H^1(K_\Sigma/K, M) = H^1(\Gal(K_\Sigma/K), M)$ for any continuous $\Gal(K_\Sigma/K)$-module $M$. 
Furthermore, the maps $\bT \to \bT/J_n\bT$ induce isomorphisms 
$H^1(K, \bT) \simeq \varprojlim_{n \geq 1} H^1_{\cF_n}(K, \bT/J_n \bT)$ and 
$H^1_{\cF^*}(K, \bT^*(1)) \simeq H^1_{\cF_\Lambda^*}(K, \bT^*(1))$. 
Note that
$H^1(K_{\Sigma}/K, \bT)$, $H^2(K_{\Sigma}/K, \bT)$, and $H^{1}_{\cF^{*}_{\Lambda}}(K, \bT^*(1))^\vee$ 
are finitely generated $\Lambda$-modules (See \cite[Section 3]{PR}). 
\end{remark}

\begin{remark}\label{free}
Note that the map 
$H^1(K, \bT) \otimes_{\Lambda} \Lambda/(X) \to H^1(K_\Sigma/K, \cT)$
induced by the map $\bT \to \cT$ is injective. 
Since $\cT$ satisfies the condition (H.1), 
the module $H^1(K_\Sigma/K, \cT)$ is a free $\cO$-module, which implies that 
so is $H^1(K, \bT) \otimes_{\Lambda} \Lambda/(X)$. 
Let $s = {\rm rank}_\cO \left( H^1(K, \bT) \otimes_{\Lambda} \Lambda/(X) \right)$. 
Since $\cT^{G_K} = 0$, the map $H^1(K, \bT) \xrightarrow{\times X} H^1(K, \bT)$ is injective. 
Thus we have a map of exact sequences:
\begin{align*}
\xymatrix{
0 \ar[r] & \Lambda^s \ar[r]^{\times X} \ar[d]^{j} & \Lambda^s \ar[r] \ar[d]^{j} & \cO^s \ar[r] \ar[d]^{\simeq} & 0 
\\
0 \ar[r] & H^1(K, \bT) \ar[r]^{\times X} & H^1(K, \bT) \ar[r] & H^1(K, \bT) \otimes_{\Lambda} \Lambda/(X) \ar[r] & 0.  
}
\end{align*}
Hence we have $X \ker(j) = \ker(j)$ and $X {\rm coker}(j) = {\rm coker}(j)$. 
Hence the map $j$ is an isomorphism and 
$H^1(K, \bT)$ is a free $\Lambda$-module of rank $s$. 
\end{remark}

\begin{proposition}\label{freerank}
Let $\fq$ be a height one prime of $\Lambda$ with $\fq \neq p\Lambda$ and 
let $S_\fq$ denote the integral closure of $\Lambda/\fq$ in its field of fractions. 
If $H^2(K_\Sigma/K, \bT)[\fq]$ is finite, 
then we have 
\begin{align*}
{\rm rank}_\Lambda \left( H^1(K, \bT) \right) 
= r + {\rm corank}_{S_\fq} \left( H^1_{\cF_{\rm can}^*}(K, (\cT \otimes S_\fq)^*(1)) \right) 
\end{align*}
where $\cF_{\rm can}$ is the canonical Selmer structure defined in Example \ref{cansel}. 
\end{proposition}
\begin{proof}
Note that $\cT \otimes_{\cO} S_\cQ$ satisfies the conditions (H.1), (H.2), and (H.3). 
Since $H^2(K_\Sigma/K, \bT)[\fq]$ is finite, by the same proof of \cite[Proposition 5.3.14]{MR1}, 
one can show that the map 
\begin{align*}
H^1(K, \bT) \otimes_{\Lambda} \Lambda/\fq \to H^1_{\cF_{\rm can}}(K, \cT \otimes_{\cO} S_{\fq})
\end{align*}
induced by the map $\bT \to \bT/\fq \bT \to \cT \otimes_\cO S_\fq$ is injective and 
its cokernel is finite. 
Since $H^1(K, \bT)$ is a free $\Lambda$-module by Remark \ref{free}, we have 
\begin{align*}
{\rm rank}_{\Lambda} \left( H^1(K, \bT) \right)
= {\rm rank}_{S_\fq} \left( H^1_{\cF_{\rm can}}(K, \cT \otimes_\cO S_\fq) \right). 
\end{align*}
Since $H^2(K_\fp, \cT) = 0$ for any prime $\fp \mid p$ of $K$, 
the core rank of $\cF_{\rm can}$ on $S_\fq$ is $r$ by Example \ref{cancore}. 
Thus by the same proof of \cite[Corollary 5.2.6]{MR1}, we have 
\begin{align*}
{\rm rank}_{S_\fq} \left( H^1_{\cF_{\rm can}}(K, \cT \otimes_\cO S_\fq) \right) 
= r + {\rm corank}_{S_\fq} \left(H^1_{\cF_{\rm can}^*}(K, (\cT \otimes_\cO S_\fq)^*(1)) \right). 
\end{align*}
\end{proof}

\begin{proposition}\label{rankleop}
The $\Lambda$-module $H^1_{\cF_{\Lambda}^*}(K, \bT^*(1))^\vee$ is torsion if and only if 
$H^1(K, \bT)$ is a free $\Lambda$-module of rank $r$. 
\end{proposition}
\begin{proof}
Since $H^2(K_\Sigma/K, \bT)$ is a finitely generated $\Lambda$-module, 
there is a hight one prime $\fq \neq p \Lambda$ of $\Lambda$ such that $H^2(K_\Sigma/K, \bT)[\fq]$ is finite. 
Then by the same proof of \cite[Proposition 5.3.14]{MR1}, 
we show that the kernel and the cokernel of the map 
\begin{align*}
H^1_{\cF_{\Lambda}^*}(K, \bT^*(1))^\vee \otimes_{\Lambda} \Lambda/\fq \to 
H^1_{\cF_{\rm can}^*}(K, (\cT \otimes_{\cO} S_{\fq})^*(1))^\vee
\end{align*}
induced by the map $(\cT \otimes_{\cO} S_{\fq})^*(1) \to (\bT/\fq \bT)^*(1) \to \bT^*(1)$ are finite. 
Thus if $H^1(K, \bT)$ is free of rank $r$, then 
$H^1_{\cF_{\Lambda}^*}(K, \bT^*(1))^\vee \otimes_{\Lambda} \Lambda/\fq$ is finite by Proposition \ref{freerank}.  
Hence $H^1_{\cF_{\Lambda}}(K, \bT^*(1))^\vee$ is a torsion $\Lambda$-module. 
If $H^1_{\cF_{\Lambda}^*}(K, \bT^*(1))^\vee$ is a torsion $\Lambda$-module, 
then there is a height one prime $\fq \neq p \Lambda$ of $\Lambda$ such that 
both $H^1_{\cF_{\Lambda}^*}(K, \bT^*(1))^\vee \otimes_{\Lambda} \Lambda/\fq$ 
and $H^2(K_\Sigma/K, \bT)[\fq]$ are finite. 
Thus by Remark \ref{free} and Proposition \ref{freerank}, $H^1(K, \bT)$ is a free $\Lambda$-module of rank $r$. 
\end{proof}

\begin{remark}
The assertion that $H^1_{\cF^*_{\Lambda}}(K, \bT^*(1))^\vee$ is a torsion $\Lambda$-module is a form 
of the weak Leopoldt conjecture. 
\end{remark}

\begin{lemma}\label{exact}
Let $n$ be a positive integer. 
Then we have an exact sequence
\begin{align*}
0 \to H^2(K_\Sigma/K, \bT)[J_n] &\to H^1(K, \bT) \otimes_\Lambda \Lambda/J_n \to H^1(K_\Sigma/K, \bT/J_n \bT). 
\end{align*}
\end{lemma}
\begin{proof}
By using the exact sequence $0 \to \bT \xrightarrow{\times X^{\beta}} \bT \to \bT/ X^\beta \bT \to 0$, 
we obtain an exact sequence
\begin{align*}
0 \to H^1(K_{\Sigma}/K, \bT) \otimes_\Lambda \Lambda/X^n \Lambda \to H^1(K_{\Sigma}/K, \bT/X^n \bT) \to 
H^2(K_{\Sigma}/K, \bT)[X^n] \to 0. 
\end{align*}
Applying $- \otimes_{\Lambda/(X^n)} \Lambda/J_n$ to the above exact sequence, 
we get the following exact sequence:
\begin{align*}
H^1(K_{\Sigma}/K, \bT/X^n \bT)[p^n] &\to H^2(K_{\Sigma}/K, \bT)[J_n] 
\to H^1(K, \bT) \otimes_\Lambda \Lambda/J_n 
\to H^1(K_{\Sigma}/K, \bT/X^n \bT) \otimes_{\Lambda/(X^n)} \Lambda/J_n.   
\end{align*}
By using the exact sequence $0 \to \bT/X^n \bT \xrightarrow{\times p^n} \bT/X^n\bT \to \bT/J_n \bT \to 0$, 
we have an exact sequence 
\begin{align*}
0 \to H^1(K_{\Sigma}/K, \bT/X^n \bT) \xrightarrow{\times p^n}  
H^1(K_{\Sigma}/K, \bT/X^n \bT) \to H^1(K_{\Sigma}/K, \bT/J_n \bT) 
\end{align*}
since we assume that $(\bT/(\varpi, X)\bT)^{G_K} = (\cT/\varpi \cT)^{G_K} = 0$ 
where $\varpi$ is a uniformizer of $\cO$. Thus we get the desired exact sequence. 
\end{proof}

The following theorem is a generalization of \cite[Theorem 5.3.6]{MR1} and \cite[Theorem 5.3.10]{MR1}. 
However, the proof of this theorem is different from the proof in \cite{MR1}. 

\begin{theorem}\label{lambda-stark}
If $\epsilon$ is a primitive $\Lambda$-adic Stark system and $i$ is a non-negative integer, 
then we have 
\begin{align*}
I_i(\epsilon) = {\rm Fitt}_{i, \Lambda} \left( H^1_{\cF_\Lambda^*}(K, \bT^*(1))^\vee \right). 
\end{align*}
In particular, $H^1_{\cF_\Lambda^*}(K, \bT^*(1))^\vee$ 
is a torsion $\Lambda$-module if and only if 
there is a $\Lambda$-adic Stark system $\eta = \{ \eta^{(n)} \}_{n \geq 1}$ 
such that $\eta_1^{(n)} \neq 0$ for some positive integer $n$. 
\end{theorem}
\begin{proof}
The first assertion follows from $H^1_{\cF^*}(K, \bT^*(1)) \simeq H^1_{\cF_\Lambda^*}(K, \bT^*(1))$ and 
Theorem \ref{main2}. 
Since 
\begin{align*}
{\rm Ann}_{\Lambda}(M)^m \subseteq {\rm Fitt}_{0, \Lambda}(M) \subseteq {\rm Ann}_{\Lambda}(M)
\end{align*} 
for any $\Lambda$-module $M$ which is generated by $m$ elements, 
$H^1_{\cF_\Lambda^*}(K, \bT^*(1))^\vee$ is a torsion $\Lambda$-module if and only if $I_0(\epsilon) \neq 0 $. 
Furthermore, by Proposition \ref{surj} (2), 
there is a $\Lambda$-adic Stark system 
$\eta = \{ \eta^{(n)}\}_{n \geq 1}$ such that $\eta_1^{(n)} \neq 0$ for some positive integer $n$ 
if and only if $I_0(\epsilon) \neq 0 $. 
This completes the proof. 
\end{proof}

We will show that Theorem \ref{lambda-stark} implies 
\cite[Theorem 5.3.6]{MR1} and \cite[Theorem 5.3.10]{MR1}. 
Assume $r = 1$. 
Let $n$ be a positive integer. 
By Matlis duality, the map $l_1 \colon H^1_{\cF_n}(K, \bT/J_n \bT) \to 
\bigcap^{1}_{\Lambda/J_n}H^1_{\cF_n}(K, \bT/J_n \bT)$ is an isomorphism. 
Hence we have a map 
\begin{align*}
\Theta^{(n)} \colon {\bf SS}_1(\bT/J_n \bT, \cF_n, \cP_{n^2}(\cF_n)) \to 
\bigcap^{1}_{\Lambda/J_n}H^1_{\cF_n}(K, \bT/J_n\bT) \xrightarrow{l_1^{-1}} H^1_{\cF_n}(K, \bT/J_n\bT) 
\end{align*}
where the first map is a projection map. 
It is easy to see that the maps $\Theta^{(n)}$ induce a map $\Theta \colon {\bf SS}_1(\bT) \to H^1(K, \bT)$. 

\begin{proposition}\label{mr}
Let $H^2(K_{\Sigma}/K, \bT)_{\rm fin}$ be the maximal finite $\Lambda$-submodule of $H^2(K_{\Sigma}/K, \bT)$. 
Suppose that $r = 1$. 
If ${\rm im}\left( \Theta \right) \neq 0$, then $H^1(K, \bT)$ is free of rank one and 
$H^1_{\cF^*_\Lambda}(K, \bT^*(1))^\vee$ is a torsion $\Lambda$-module. 
Furthermore, we have  
\begin{align*}
{\rm Fitt}_{0, \Lambda} \left( H^1_{\cF_{\Lambda}^*}(K, \bT^*(1))^\vee \right) 
= {\rm char}_\Lambda \left( H^1(K,\bT)/\Lambda \Theta(\epsilon) \right) 
 {\rm Ann}_{\Lambda}\left( H^2(K_{\Sigma}/K, \bT)_{\rm fin} \right).
\end{align*}
for any primitive $\Lambda$-adic Stark system $\epsilon$. 
In particular, 
\begin{align*}
{\rm char}_\Lambda \left( H^1_{\cF_{\Lambda}^*}(K, \bT^*(1))^\vee \right) = 
{\rm char}_\Lambda \left( H^1(K,\bT)/\Lambda \Theta(\epsilon) \right). 
\end{align*}
\end{proposition}
\begin{proof}
Since the map 
$l_1 \colon H^1_{\cF_n}(K, \bT/J_n \bT) \to \bigcap^1_{\Lambda/J_n} H^1_{\cF_n}(K, \bT/J_n \bT)$ 
is an isomorphism for any positive integer $n$ and ${\rm im}\left( \Theta \right) \neq 0$, 
the $\Lambda$-module $H^1_{\cF_{\Lambda}^*}(K, \bT^*(1))^\vee$ is a torsion by Theorem \ref{lambda-stark}. 
Thus $H^1(K, \bT)$ is a free $\Lambda$-module of rank one by Proposition \ref{rankleop}. 
Let $\epsilon = \{\epsilon^{(n)}\}_{n \geq 1}$ be a primitive $\Lambda$-adic Stark system. 
Put ${\rm Ind}(\epsilon) = {\rm char}_\Lambda \left( H^1(K,\bT)/\Lambda \Theta(\epsilon) \right)$. 
Then we have 
\begin{align*}
{\rm Ind}(\epsilon) = \{ f(\Theta(\epsilon)) \mid f \in \Hom_{\Lambda}(H^1(K, \bT), \Lambda) \}. 
\end{align*}
Let $n$ be a positive integer and $\Lambda_n = \Lambda/J_n$. 
We will compute the ideal $I_0(\epsilon^{(n)}) \subseteq \Lambda_n$. 
Let $M_n$ be the image of the map $H^1(K, \bT) \to H^1_{\cF_n}(K, \bT/J_n \bT)$. 
Fix an isomorphism $H^1(K, \bT) \simeq \Lambda$. 
Let $Z_n$ be the image of the injective map 
$H^2(K_{\Sigma}/K, \bT)[J_n] \to H^1(K, \bT) \otimes_\Lambda \Lambda_n \simeq \Lambda_n$. 
By Lemma \ref{exact}, the fixed isomorphism $H^1(K, \bT) \simeq \Lambda$ induces an isomorphism 
$M_{n}^* \simeq \Lambda_n[Z_n]$. 
Let $\Theta(\epsilon)_n = \Theta(\epsilon) \bmod J_n \in H^1_{\cF_n}(K, \bT/J_n\bT)$. 
Since the map $H^1_{\cF_n}(K, \bT/J_n\bT)^* \to M_{n}^*$ is surjective 
and $\Theta(\epsilon)_{n} \in M_n$, we have
\begin{align*}
I_0(\epsilon^{n}) = \{ f(\Theta(\epsilon)_{(n)}) \mid f \in M_{n}^* \} 
= {\rm Ind}(\epsilon) \Lambda_n[Z_n].
\end{align*}
By the definition of the ideal $Z_n$, we have 
\begin{align*}
\ker \left( \Lambda \to \Lambda_n / \Lambda_n[Z_n] \right)
= {\rm Ann}_{\Lambda}\left( H^2(K_{\Sigma}/K, \bT)[J_n] \right). 
\end{align*}
Since the $\Lambda$-module $H^2(K_{\Sigma}/K, \bT)$ is of finite type, we have 
\begin{align*}
\ker \left( \Lambda \to \Lambda_{n}/\Lambda_{n}[Z_n] \right) 
= {\rm Ann}_{\Lambda}\left( H^2(K_{\Sigma}/K, \bT)_{\rm fin} \right) 
\end{align*}
for all sufficiently large integer $n$. 
Thus we conclude that 
\begin{align*}
I_0(\epsilon) 
= \varprojlim_{n \geq 1}{\rm Ind}(\epsilon) \Lambda_{n}[Z_n] 
= \bigcap_{n \geq 1} \left({\rm Ind}(\epsilon) 
{\rm Ann}_{\Lambda}\left( H^2(K_{\Sigma}/K, \bT)_{\rm fin} \right) + J_n \right) 
= {\rm Ind}(\epsilon) {\rm Ann}_{\Lambda}\left( H^2(K_{\Sigma}/K, \bT)_{\rm fin} \right). 
\end{align*}
Hence this proposition follows from Theorem \ref{lambda-stark}. 
\end{proof}

\begin{proposition}\label{lastthm}
For a positive integer $n$, let $i_n \colon \bigwedge^r_\Lambda H^1(K, \bT) \otimes_\Lambda \Lambda/J_n \to 
\bigcap^r_{\Lambda/J_n} H^1_{\cF_n}(K, \bT/J_n\bT)$ 
denote the canonical map. 
If $H^2(K_\Sigma/K, \bT)_{\rm fin} = 0$, there is a unique $\Lambda$-module homomorphism 
\begin{align*}
\Theta \colon {\bf SS}_r(\bT) \to \bigwedge^r_\Lambda H^1(K, \bT) 
\end{align*} 
such that $i_n \circ \Theta(\epsilon) = \epsilon^{(n)}_1$ for any positive integer $n$ and 
$\Lambda$-adic Stark system $\epsilon = \{\epsilon^{(n)} \}_{n \geq 1}$. 
Furthermore, if ${\rm im}\left( \Theta \right) \neq 0$, then $H^1(K, \bT)$ is free of rank $r$, 
$H^1_{\cF^*}(K, \bT^*(1))^\vee$ is a torsion $\Lambda$-module, and 
\begin{align*}
{\rm Fitt}_{0, \Lambda} \left( H^1_{\cF^*_{\Lambda}}(K, \bT^*(1))^\vee \right) = 
{\rm char}_\Lambda \left( \bigwedge^r_\Lambda H^1(K, \bT) \bigg/ \Lambda \Theta(\epsilon) \right)  
\end{align*}
for any primitive $\Lambda$-adic Stark system $\epsilon$. 
\end{proposition}
\begin{proof}
Let $n$ be a positive integer and $s = {\rm rank}_\Lambda \left( H^1(K, \bT) \right)$. 
Note that $s \geq r$ by Proposition \ref{freerank}. 
Let $\Lambda_n = \Lambda/J_n$. 
By Lemma \ref{exact} and $H^2(K_\Sigma/K, \bT)_{\rm fin} = 0$, 
the map $H^1(K, \bT) \otimes_\Lambda \Lambda_n \to H^1_{\cF_n}(K, \bT/J_n\bT) $
is a split injection. 
Thus $H^1_{\cF_n}(K, \bT/J_n\bT)$ has a free $\Lambda_n$-submodule of rank $s$ and 
the map $i_n$ is injective. 
Let $\epsilon = \{\epsilon^{(n)} \}_{n \geq 1}$ be a $\Lambda$-adic Stark system. 
By Proposition \ref{Fitt}, there is a unique element 
$\Theta(\epsilon)_n \in \bigwedge^{r}_\Lambda H^1(K, \bT) \otimes_\Lambda \Lambda_n$ 
such that $i_n \circ \Theta(\epsilon)_n = \epsilon^{(n)}_1$. 
By the proof of Proposition \ref{Fitt}, 
the set $\{ \Theta(\epsilon)_{n} \}_{n \geq 1}$ becomes an inverse system. 
Thus we have the desired map 
$\Theta(\epsilon) := \varprojlim_{n \geq 1} \Theta(\epsilon)_n$. 
Let $\epsilon$ be a primitive $\Lambda$-adic Stark system. 
If ${\rm im}(\Theta) \neq 0$, we have $s = r$ by Proposition \ref{Fitt}. 
Hence $H^1_{\cF^*_\Lambda}(K, \bT^*(1))^\vee$ is a torsion $\Lambda$-module by Proposition \ref{rankleop}. 
Furthermore, we have  
\begin{align*}
{\rm Fitt}_{0, \Lambda} \left( \bigwedge^r_\Lambda H^1(K, \bT) \bigg/ \Lambda \Theta(\epsilon) \right)\Lambda_n 
&= {\rm Fitt}_{0, \Lambda_n} \left( \bigwedge^r_\Lambda H^1(K, \bT) \otimes_\Lambda \Lambda_n 
\bigg/ \Lambda_n \Theta(\epsilon)_n \right) 
\\
&= {\rm Fitt}_{0, \Lambda_n} \left( H^1_{\cF_{n}^*}(K, (\bT/J_n\bT)^*)^* \right) 
\\
&= {\rm Fitt}_{0, \Lambda} \left( H^1_{\cF_{\Lambda}^*}(K, \bT^*(1))^\vee \right) \Lambda_n
\end{align*}
where the second equality follows from Proposition \ref{Fitt} and 
the third equality follows from Remark \ref{dual} and Lemma \ref{dmult}. 
Since $\Lambda$ is a complete noetherian local ring, this completes the proof. 
\end{proof}

\begin{remark}\label{buk}
We use the same notation as in Example \ref{cyclo}. 
If we assume Rubin-Stark conjecture, then 
we have a primitive $\Lambda$-adic stark system $\epsilon^{\text{R-S}}$ 
(See \cite[Section 4.3]{Bu}, \cite{bus}, and Example \ref{cyclo}). 
Since $H^2(K_\Sigma/K, \bT)_{\rm fin} = 0$, we have 
\begin{align*}
{\rm char} \left( H^1_{\cF^*_{\Lambda}}(K, \bT^*(1))^\vee \right) = 
{\rm char} \left( \bigwedge^r_{\Lambda} H^1(K, \bT) \bigg/ \Lambda \Theta(\epsilon^{\text{R-S}}) \right) 
\end{align*}
by Proposition \ref{lastthm}. 
Let $H^1(K_p, \bT) = \bigoplus_{\fp \mid p}H^1(K_\fp, \bT)$ and 
${\rm loc}_p \colon H^1(K, \bT) \to H^1(K_p, \bT)$ denote the localization map at $p$. 
Suppose that the map ${\rm loc}_p$ is injective. 
Then by Theorem \ref{pt}, we obtain an exact sequence 
\begin{align*}
0 \to H^1(K, \bT) \to H^1(K_p, \bT) \to H^1_{\cF^*_{\rm str}}(K, \bT^*(1))^\vee 
\to H^1_{\cF_{\Lambda}^*}(K, \bT^*(1))^\vee \to 0. 
\end{align*}
Here the Selmer structure $\cF_{\rm str}$ is defined in \cite[proposition 4.11]{Bu}. 
Then we have 
${\rm rank}_{\Lambda} \left( H^1(K, \bT) \right) = {\rm rank}_{\Lambda} \left( H^1(K_p, \bT) \right) = r$ 
by Proposition \ref{lastthm} and \cite[Proposition 4.7 (i)]{Bu}. 
Thus we get an equality
\begin{align*}
{\rm char} \left( H^1_{\cF^*_{\rm str}}(K, \bT^*(1))^\vee \right) = 
{\rm char} \left( \bigwedge^r_\Lambda H^1(K_p, \bT) \bigg/ 
\Lambda \cdot {\rm loc}_p(\Theta(\epsilon^{\text{R-S}})) \right). 
\end{align*}
Therefore we show that Theorem \ref{lastthm} implies \cite[Theorem 4.15]{Bu}. 
\end{remark}

\end{document}